\documentclass[reqno, a4paper]{amsart}
\usepackage{amssymb, mathdots}
\usepackage{mathabx}
\usepackage{mathrsfs}
\usepackage[utf8]{inputenc}
\usepackage{amsmath,  graphicx, tensor}
\usepackage{tikz}
\usetikzlibrary{matrix, arrows.meta}
\usetikzlibrary{cd}

\usepackage{enumerate}
\usepackage{hyperref}
\DeclareSymbolFont{bbold}{U}{bbold}{m}{n}
\DeclareSymbolFontAlphabet{\mathbbold}{bbold}
\def\qmod#1#2{{\hbox{}^{\displaystyle{#1}}}\!\big/\!\hbox{}_{
\displaystyle{#2}}}

 \def\psp#1#2%
  {\mathop{}%
   \mathopen{\vphantom{#2}}^{#1}%
   \kern-\scriptspace%
    \hskip -0.3mm{#2}} 
 \def\psb#1#2%
  {\mathop{}%
   \mathopen{\vphantom{#2}}_{#1}%
   \kern-\scriptspace%
    \hskip -0.0mm{#2}} 
 \def\pscr#1#2#3%
  {\mathop{}%
   \mathopen{\vphantom{#3}}^{#1}_{#2}%
   \kern-\scriptspace%
    \hskip -0.3mm{#3}} 

\def\C{{\mathbb C}}

\def\H{{\mathbb H}}

\def\N{{\mathbb N}}

\def\P{{\mathbb P}}

\def\R{{\mathbb R}}

\def\Z{{\mathbb Z}}

\def\textmap#1{\mathop{\vbox{\ialign{
                                  ##\crcr
      ${\scriptstyle\hfil\;\;#1\;\;\hfil}$\crcr
      \noalign{\kern 1pt\nointerlineskip}
      \rightarrowfill\crcr}}\;}}
\def\bigtextmap#1{\mathop{\vbox{\ialign{
                                  ##\crcr
      ${\hfil\;\;#1\;\;\hfil}$\crcr
      \noalign{\kern 1pt\nointerlineskip}
      \rightarrowfill\crcr}}\;}}
      
\newcommand{\cal}{\mathcal}
\def\textlmap#1{\mathop{\vbox{\ialign{
                                  ##\crcr
      ${\scriptstyle\hfil\;\;#1\;\;\hfil}$\crcr
      \noalign{\kern-1pt\nointerlineskip}
      \leftarrowfill\crcr}}\;}}

\def\fg{{\mathfrak f}}

\def\jg{{\mathfrak j}}

\def\lg{{\mathfrak l}}
\def\mg{{\mathfrak m}}
\def\ng{{\mathfrak n}}

\def\sg{{\mathfrak s}}
\def\tg{{\mathfrak t}}

\def\Jg{{\mathfrak J}}

\def\Ng{{\mathfrak N}}

\def\Qg{{\mathfrak Q}}

\def\Wg{{\mathfrak W}}

\theoremstyle{remark}
\newtheorem{ex}{Example}[section]

\newtheorem{sz}{Satz}[section]
\theoremstyle{plain}
\newtheorem{thry}[sz]{Theorem}
\newtheorem{pr}[sz]{Proposition}
\newtheorem{co}[sz]{Corollary}
\newtheorem{dt}[sz]{Definition}
\newtheorem{lm}[sz]{Lemma}
\newtheorem{re}[sz]{Remark}

\def\End{\mathrm {End}}
\def\Aut{\mathrm {Aut}}
\def\Spin{\mathrm {Spin}}

\def\SU{\mathrm {SU}}
\def\SO{\mathrm {SO}}

\def\GL{\mathrm {GL}}

\def\Pic{\mathrm {Pic}}

\def\Hom{\mathrm{Hom}}

\def\id{ \mathrm{id}}
\def\im{\mathrm{im}}

\def\Ah{\mathrm{Ah}}
\newcommand\smvee{{\hskip -0.1ex \raise 0.2ex\hbox{$\scriptscriptstyle\vee$}}\hskip -0,3ex}

\def\edf{\coloneq}

\begin{document}

\title{Real structures on primary Hopf surfaces}

\author{Zahraa KHALED}
\address{Aix Marseille Univ, CNRS, I2M, Marseille, France.}
\email{zahraa.khaled@univ-amu.fr}
\begin{abstract}
The first goal of this article is to give a complete classification (up to Real biholomorphisms) of Real primary Hopf surfaces $(H,s)$, and, for any such pair,  to describe in detail the following naturally associated objects : the group $\mathrm{Aut}_h(H,s)$ of Real automorphisms, the   Real Picard group $(\mathrm{Pic}(H),\hat s^*)$, and the Picard group of Real holomorphic line bundles $\mathrm{Pic}_{\mathbb{R}}(H)$. 

Our second goal:   the classification of Real primary Hopf surfaces up to equivariant diffeomorphisms, which  will allow us to describe explicitly in each case the real locus  $H(\mathbb{R})=H^s$ and the quotient $H/\langle s\rangle$.   \end{abstract}


\maketitle

\tableofcontents

 \section*{Acknowledgements} 
 
I am indebted to my PhD advisor Andrei Teleman for suggesting me the problems treated in this article and for  many helpful ideas. I also thank Karl Oeljeklaus and Anne Pichon for their advices and their interest in my results. 

\section{Introduction}\label{intro}

Let $X$ be a complex manifold, and let $J$ be the (integrable) almost holomorphic structure on its underlying differentiable manifold $X$ defining its complex structure. We will denote by $\bar X$ the complex manifold defined by $-J$. Note that the data of an anti-holomorphic isomorphism $X\to X$ is  equivalent to the data of a biholomorphism $X\to \bar X$. 

A Real structure (in the sense of Atiyah) on $X$  is an anti-holomorphic involution of $X$ \cite{At}, \cite{S}, \cite{GH}. A Real complex manifold is a pair $(X,s)$ consisting of a complex manifold and a Real structure on it.

The theory of Real complex manifolds originates from algebraic geometry, see \cite[section 1]{At}: a smooth projective variety $X\subset \P^n_\C$ defined by a system of homogeneous polynomial equations with real coefficients has a natural Real structure induced by the conjugation $\P^n_\C\to\P^n_\C$. An ample literature is dedicated to this theory in  algebraic geometric framework (see for instance \cite{S}, \cite{GH}). On the other hand not much is known on the classification of Real structures on non-algebraic manifolds. An important contribution in this direction is Paola Frediani's article \cite{Fr}, which is dedicated to the (holomorphic and topological) classification of Real  Kodaira surfaces. In this article we treat similar problems but for primary Hopf surfaces. 
\vspace{2mm}

The real locus of a Real complex manifold $(X,s)$ is just the fixed point locus $X^s$ (also denoted $X(\R)$ if $s$ has been fixed)  of its Real structure. 

Let $(X,s)$, $(Y,\sigma)$ be Real complex manifolds.  A biholomorphism $f:X\to Y$ is called Real (or compatible with the Real structures) if $\sigma\circ f=f\circ s$. The fundamental problem of the theory  is the classification of Real complex manifolds up to Real biholomorphisms.   

The group of real biholomorphisms of a Real complex manifold  $(X,s)$ is the subgroup 
$$\Aut(X,s)\coloneq \{f:X\to X|\ f \hbox{ biholomorphism, }f\circ s=s\circ f\}$$
 of the biholomorphism  group  $\Aut_h(X)$. \vspace{2mm}
 
 Let $(M,s)$ be a differentiable manifold endowed with an involution $s$ and $E$ be a complex vector bundle on $M$. We recall \cite[Section 1]{At} that   
 \begin{dt}\label{RealBdlsmooth}
 A Real structure on $E$ is a fiberwise anti-linear $s$-covering isomorphic involution $\varphi:E\to E$. A Real bundle on  $(M,s)$  is pair $(E,\phi)$ consisting of a complex bundle $E$ on $M$ and a Real structure $\phi$ on $E$.	
 \end{dt}

  Let $(X,s)$ be a Real complex manifold. 
 
 \begin{dt} \label{RealBdlsHol} A Real holomorphic bundle on $X$ is a pair $(E,\phi)$, where $E$ is a holomorphic bundle on $X$ and $\phi$ an anti-holomorphic Real structure on $E$. 
 \end{dt}

 Let $E$  be a holomorphic bundle of rank $r$ on $X$. The pull-back $s^*(\bar E)$ has a natural structure of a holomorphic bundle on $X$ (see for instance \cite[section 1.2]{OT}): it is just the pull-back of $\bar E$, regarded as a holomorphic bundle on $\bar X$, via the holomorphic map $s:X\to \bar X$.  The map $[E]\mapsto [s^*(\bar E)]$ defines a natural involution on the set of isomorphism classes of holomorphic bundles on $X$. In particular, for $r=1$, we obtain an involution 
 $$\bar s^*:\Pic(X)\to \Pic(X),\ \bar s^*([L])\coloneq [s^*(\bar L)]$$
   on the Picard group of $X$; this involution  is an anti-holomorphic group isomorphism, so $(\Pic(X),\bar s^*)$ becomes a Real complex Lie group.
 
 The definitions above allow us to associate to a compact, connected Real complex manifold two natural invariants constructed using holomorphic line bundles: 
 \begin{itemize}
 	\item The group $\Pic_\R(X)$ of isomorphism classes of Real holomorphic line bundles on $X$.
 	\item The Real complex Lie group $(\Pic(X),\bar s^*)$. \end{itemize}  
 
 Note that one has an obvious comparison real Lie group morphism 
 $$\Pic_\R(X)\to \Pic(X)(\R),\ [L,\phi]\mapsto [L],$$
  which is always injective and is an isomorphism when $X(\R)\ne\emptyset$.\\

  The goals of this article are:
  \begin{enumerate}[(G1)]
  \item \label{G1} To give complete classification of Real primary Hopf surfaces (up to Real biholomorphisms) with an explicit description of the set of isomorphism classes. 
  \item To describe explicitly, for any Real primary Hopf surface $(H,s)$, the following naturally associated objects: 
  \begin{enumerate}
  \item its automorphism group $\Aut(H,s)\subset \Aut_h(H)$.
  \item its Real Picard group $(\Pic(H),\bar s^*)$ of holomorphic line bundles, and   its Picard group $\Pic_\R(H)$ of Real holomorphic line bundles.
    \end{enumerate}
  \item  To classify differential-topologically  the Real primary Hopf surfaces, and, for any  Real primary Hopf surface $(H,s)$, to describe explicitly  the fixed point (real) locus $H^s$ and the quotient $H/\langle s\rangle$.
 \end{enumerate}
 \vspace{2mm} 
For (G\ref{G1}), recall first that \cite{BHPV}:
\begin{dt}
A primary Hopf surface is a compact complex surface $H$ whose universal covering is biholomorphic to $W\coloneq \C^2\setminus\{0\}$, and whose fundamental group is isomorphic to $\Z$.
\end{dt}
From this definition, it follows that any primary Hopf surface is biholomorphic to a quotient of the form 
$$H_f={\qmod{W}{\langle f\rangle}} \ ,$$
where $\langle f\rangle$ is the cyclic group generated by a biholomorphism $f\in \Aut_h(W).$ 
By a fundamental theorem of Kodaira \cite{Ko1}, it follows that any primary Hopf surface is biholomorphic to $ {W}/{\langle f\rangle}$ where $f$ is a biholomorphism of the form

$$f(z,w)=(\alpha z+\lambda w^n,\beta w)$$

where 
$$0<\vert\alpha\vert \leq\vert\beta\vert<1 , \ n\in \N  ,\ \lambda(\alpha-\beta^n)=0.   $$
If the coefficients of $f$ are real, the standard conjugation $c:W\to W$ will obviously descend to a Real structure on $H_f$. We will see that there exists interesting classes of Real primary Hopf surfaces which are  not of this type. Moreover, there exists  Real primary Hopf surfaces defined by holomorphic contractions $f$ whose coefficients are not real.

Note first that Kodaira's  theorem does not give a precise classification of primary Hopf surfaces, because it is not clear under which conditions the surfaces associated with two 4-tuples $(\alpha,\beta,\lambda,n)$, $(\alpha',\beta',\lambda',n')$ as above are biholomorphic.
 Following  \cite{We} we introduce five classes of holomorphic contractions:
\begin{equation}\label{WehlerCLasses}
\begin{split}
{IV}\coloneq& \left\{f:W\to W|\ f\begin{pmatrix}
z\\w 	
\end{pmatrix}
=\begin{pmatrix}\alpha z\\ \alpha w\end{pmatrix}\vline\ 0< |\alpha| <1 \right\},
\\	
 {III}\coloneq& \left\{f:W\to W|\ f\begin{pmatrix}
z\\w 	
\end{pmatrix}=\begin{pmatrix}\delta^r z\\ \delta w\end{pmatrix}\vline\ \ r\in \N_{\geq 2},\ 0<|\delta|<1  \right\},
\\
 {II}_a\coloneq &\left\{f:W\to W|\ f\begin{pmatrix}
z\\w 	
\end{pmatrix}=\begin{pmatrix}\delta^r z+w^r\\ \delta w\end{pmatrix}\vline\  r\in \N_{\geq 2},\ 0<|\delta|<1  \right\},
\\
 {II}_b\coloneq &\left\{f:W\to W|\ f\begin{pmatrix}
z\\w 	
\end{pmatrix}=\begin{pmatrix}\alpha z+w\\ \alpha w\end{pmatrix}\vline\   \ 0<|\alpha|<1  \right\},
\\
 {II}_c\coloneq &\left\{f:W\to W|\ f\begin{pmatrix}
z\\w 	
\end{pmatrix}=\begin{pmatrix}\alpha z\\ \delta w\end{pmatrix}\vline\  \begin{array}{c} 0<|\alpha|<1\\ 0< |\delta|<1\end{array},\ \alpha\ne \delta^r\, \forall r\in\N  \right\}.
\end{split}
\end{equation}
The map  
\begin{equation}\label{maptoisotypes} IV\cup III\cup II_a\cup II_b\cup II_c\ni f\mapsto [H_f] 
\end{equation}
which assigns to a holomorphic contraction $f$ the biholomorphism class of the corresponding Hopf surface $H_f$ is surjective, but {\it not} injective. Indeed, the contractions   $f$, $f'\in II_c$ associated with the pairs $(\alpha,\delta)$, $(\alpha',\delta')=(\delta,\alpha)$ are biholomorphic. Note that this exception to injectivity is not mentioned in \cite{We}. In fact in \cite{We} the class $II_c$ is defined by imposing the additional condition $|\alpha|<|\delta|$. Unfortunately with this restrictive definition of   $II_c$ {\it one loses the surjectivity of the map} (\ref{maptoisotypes}), because biholomorphism types of Hopf surfaces associated with pairs $(\alpha,\delta)$ satisfying 
\begin{equation}\label{subclIIc}
0<|\alpha|=|\delta|<1,\ \alpha\ne \delta^r\, \forall r\in\N
\end{equation}
will not belong to its image. This remark is important for our purposes because precisely in the subclass of $II_c$ defined by (\ref{subclIIc}) -- the subclass which is omitted in \cite{We} -- we will find contractions $f$ for which $H_f$ admits Real structures although the coefficients of $f$ are {\it not} real.   

Our first step in the classification of Real structures on primary Hopf surfaces is to divide them in two classes:   a  Real structure $\phi$ on $H_f$ will be called {\it even} ({\it odd}) if it admits a lift $\hat \phi:W\to W$ with   $\phi^2=\id_W$ (respectively   $\phi^2=f$). The even (odd) Real structures are classified by Theorem \ref{ClassEven} (respectively Theorem \ref{ClassOdd}).
\vspace{2mm}
   
Concerning (G2) we will give explicit descriptions of the automorphism group $\Aut(X,s)$ of all Real primary Hopf surface. For instance, when $f\in IV$ with negative coefficient $\alpha$ we obtain $\Aut(H_f,\sg_f)\simeq\Spin^c(3)$, where $\sg_f$ denotes  the canonical odd Real structure  on $H_f$ (see Corollary \ref{Spinc(3)}).

Our results concerning the Real complex group $(\Pic(X),\bar s^*)$  and the group $\Pic_\R(X)$  of a Real primary Hopf surface $(X,s)$ are (see Proposition \ref{Pic(X)R}): 
\begin{enumerate}
\item $(\Pic(X),\bar s^*)$ is 	always isomorphic to $(\C^*,\bar{\ } )$.
\item The canonical monomorphism $\Pic_\R(X)\to \Pic(X)(\R)=\R^*$ is an isomorphism if $(X,s)$ is even and identifies $\Pic_\R(X)$ with $\R_{>0}$ if $(X,s)$ odd. 
\end{enumerate}
Our results for  the goal (G3) give a complete differential topological classification of Real primary Hopf surfaces (see Theorems \ref{ClassDiffEven}, \ref{ClassDiffOdd} and Remark \ref{mu-mu0}).  The final result is: 
\begin{itemize}
\item 	Any even Real primary Hopf surface is equivariantly diffeomorphic to either 
$$\big(S^1\times  S^3, (\zeta,(u,v))\mapsto (\zeta,(\bar u,\bar v))\big),$$
 or 
 $$\big(S^1\times  S^3, (\zeta,(u,v))\mapsto (\zeta,(\bar u,\zeta\bar v))\big).$$
\item Any odd Real primary Hopf surface is equivariantly diffeomorphic to  
$$\big (S^1\times  S^3, (\zeta,Z)\mapsto (-\zeta,Z)\big).$$
\end{itemize}

Taking into account the results of section \ref{RealLocusSection}, this shows that the equivariant differential topological type of a Real primary Hopf surface is determined by the type (even or odd) and the orientability of the real locus.

The main idea in the proof of this classification result  is:   for a contraction $f\in IV\cup III\cup II_a\cup II_b\cup II_c$ with Real coefficients and positive diagonal coefficients, we  construct 1-parameter group of diffeomorphisms  $(f^t)_{t\in\R}$  of $W$ acting freely on $W$ such that $f=f^1$.   Moreover we also construct a compact 3-dimensional submanifold $\Sigma\subset W$ which is transversal to the orbits of this group and can be identified to $S^3$  via a diffeomeorphism which commutes with the conjugation and the involutions $(z,w)\mapsto (\pm z,\pm w)$.  
\vspace{1mm} 

Finally will show that: \begin{itemize}
	\item The real locus $X^s$ of an even Real primary Hopf surface $(X,s)$ is either a torus, or a Klein bottle, whereas the real locus of an odd Real primary Hopf surface is always empty.
	\item     The quotient $X/\langle s\rangle$ associated with a Real primary Hopf surface $(X,s)$ is always homeomorphic to $S^1\times S^3$, and we describe the position of the fixed point locus $X^s$ in this quotient.
	
\end{itemize}

Note that, by the equivariant slice theorem,  for any Real complex surface $(X,s)$, the quotient $X/\langle s\rangle$ is a topological 4-manifold.

\section{Holomorphic and anti-holomorphic automorphisms}\label{sect1}

A fundamental role in this article will be played by the results of Wehler on the classification of primary Hopf surfaces and their automorphism group. In this section we review these results and we continue with the classification of the anti-holomorphic automorphisms of primary Hopf surfaces.

\subsection{Wehler's classification of primary Hopf surfaces}\label{WehlClass}

A precise classification of primary Hopf surfaces -- with explicit descriptions of the automorphism groups -- has been given by Wehler \cite{We}. His result can be formulated as follows: 
\begin{thry}\label{Wehl}
Consider the sets ${IV}$, ${III}$, ${II_a}$,  ${II_b}$, ${II_c}\subset\Aut_h(W)$ defined in (\ref{WehlerCLasses}). 

\begin{enumerate}

\item  For every primary Hopf surface $H$ there exists  
$f\in  {IV}\cup  {III}\cup  {II_a}\cup  {II_b}\cup  {II_c}$
such that $H\simeq H_f$.   

\item \label{no-inj} For $f$, $f'\in {IV}\cup  {III}\cup  {II_a}\cup  {II_b}\cup  {II_c}$ we have $H_f\simeq H_{f'}$ if and only if either $f=f'$, or $f$ and $f'$ belong to $II_c$, and the corresponding coefficients $\alpha$, $\delta$, $\alpha'$, $\delta'$ satisfy   $\alpha'=\delta$, $\delta'=\alpha$.
\item  For any  $f$ the group  $\Aut_h(W)^f$ of holomorphic automorphisms of $W$  commuting with $f$ is given by the table below:

$$  \begin{array}  {|c|c|}
\hline   & \\ [-0.8em]
\hbox{The class of } f &  \Aut_h(W)^f   
 \\ [0.1em]
 \hline  & \\ [-0.8em]
IV & \GL(2,\C)  
\\ [0.2em]
 \hline  & \\ [-0.8em]
III &  \left\{\begin{pmatrix}
z\\w 	
\end{pmatrix}\mapsto\ \begin{pmatrix}az+bw^r\\dw\end{pmatrix}\vline\  a\in\C^*,\ d\in\C^*, b\in \C\right\}    
\\ [0.8em]
 \hline  & \\ [-0.8em]
II_a  &\left\{\begin{pmatrix}
z\\w 	
\end{pmatrix}\mapsto\ \begin{pmatrix}a^rz+bw^r\\aw\end{pmatrix}\vline\  a\in \C^*, b\in \C\right\}     
\\ [0.8em]
 \hline  & \\ [-0.8em]
II_b &  \left\{\begin{pmatrix}
z\\w 	
\end{pmatrix}\mapsto\ \begin{pmatrix}az+bw\\aw\end{pmatrix}\vline\  a\in \C^*, b\in  \C\right\}   
\\ [0.8em]
\hline  & \\ [-0.8em]
II_c & \left\{\begin{pmatrix}
z\\w 	
\end{pmatrix}\mapsto\ \begin{pmatrix}az\\dw\end{pmatrix}\vline\  a\in\C^*,\ d\in\C^*\right\}   
\\ [0.8em]
\hline
\end{array}
$$
\item In each case the cyclic group $\langle f\rangle$ is a central subgroup of $\Aut_h(W)^f$, and the automorphism group $\Aut_h(H_f)$ is identified with   $\Aut_h(W)^f/\langle f\rangle$.
\end{enumerate}
\end{thry}

Therefore the name of the class gives the dimension of the automorphism group of the corresponding surface.

\begin{re}
 In \cite{We} the case $|\alpha|=|\delta|$ is omitted in the definition of $II_c$, so the exception (\ref{no-inj}) to the injectivity of the map 
$$ {IV}\cup  {III}\cup  {II_a}\cup  {II_b}\cup  {II_c}\ni f\mapsto [H_f]$$ is not mentioned either.
\end{re}

\subsection{Anti-holomorphic automorphisms of  primary Hopf surfaces }

We start with a general remark about topological automorphisms of Hopf surfaces:

\begin{pr}\label{hatsigma}

Let $H=H_f= {W}/{\langle f \rangle}$ be a primary Hopf surface and let $ \pi:W\to\ H$ be the canonical map.  Let $\sigma:H\to\ H$ be a homeomorphism. Then 
\begin{enumerate}
\item There exists a homeomorphism $\hat{\sigma}:W\to\ W$ such that $\pi \circ \hat{\sigma}=\sigma \circ \pi$.
\item For any such homeomorphism $\hat{\sigma}$ we have $\hat{\sigma}\circ\ f \circ\ \hat{\sigma}^{-1}\in\{f,f^{-1}\}$. 
\end{enumerate}
\end{pr}

\begin{proof}

\begin{equation}
\begin{tikzcd}[row sep=large, column sep=large]
W\ar[d,"\pi"']\ar[dr, "\sigma\circ \pi" description]\ar[r, dashed, "\hat\sigma"] & W\ar[d,"\pi"]\\
H\ar[r,"\sigma"] & H	
\end{tikzcd}	
\end{equation}

(1) The composition $\sigma\circ \pi$   remains a covering  and, since $W$ is simply connected, the uniqueness theorem of the universal covering,  guarantees the existence of a homeomorphism $\hat{\sigma}:W\to\ W$ verifying the equality 
$$\pi\circ\hat{\sigma} = \sigma \circ \pi.$$ 
(2) The group
$$\Aut_{H}(W)=\{g:W\to\ W|\ g \hbox{ homeomorphism},\ \pi\circ\ g=\pi\}$$
 of topological automorphisms of the universal covering $\pi$ (of deck transformations) coincides with the cyclic group $\langle  f\rangle$. On the other hand the map
 $$
 g\mapsto \hat{\sigma}\circ g \circ \hat{\sigma}^{-1}
 $$
 is a group automorphism of $\Aut_{H}(W)$, so it coincides either with $\id_{\Aut_{H}(W)}$ or with the automorphism $g\mapsto g^{-1}$.  Replacing $g$ by $f$ we obtain $
 \hat{\sigma}\circ f \circ \hat{\sigma}^{-1}\in\{f, f^{-1}\}$ as claimed.
 \end{proof}  
 
In the case when $\sigma:H \to\ H$ is holomorphic or anti-holomorphic we have a more precise result:
  
\begin{pr}\label{PrHolAnti}
Let $\sigma:H \to\ H$ be a holomorphic (anti-holomorphic) automorphism of $H=H_f$. Then
\begin{enumerate}
	\item There exists a a holomorphic (anti-holomorphic) automorphism $\hat \sigma$ of $W$ such that $\pi\circ\hat \sigma=\sigma\circ\pi$.
	\item For any such  automorphism $\hat{\sigma}$ we have $\hat{\sigma}\circ\ f \circ\ \hat{\sigma}^{-1}=f$.
\end{enumerate}
	
\end{pr}	
\begin{proof}
Since $\pi$ is locally biholomorphic, it follows that $\hat{\sigma}$ is 	 holomorphic (anti-holomorphic) if $\sigma$ is  holomorphic (anti-holomorphic). By Hartogs theorem (applied to $\hat{\sigma}$ or to its composition with the conjugation automorphism) it follows that $\hat{\sigma}$ extends to a holomorphic (anti-holomorphic) automorphism $\tilde\sigma$ of $\C^2$ with $\tilde\sigma(0)=0$.

We can suppose that $f\in  {IV}\cup  {III}\cup  {II_a}\cup  {II_b}\cup  {II_c}$. Any such $f$ is a holomorphic contraction. It follows that for any $w_0\in W$ we have
$$
\lim_{n\to\infty} f^n(w_0)=0,\ \lim_{n\to\infty} f^{-n}(w_0)=\infty
$$
in the end compactification $W\cup\{0,\infty\}$ of $W$. Since $\hat \sigma$ extends to a homeomorphism $\tilde\sigma$ of $\C^2$ with $\tilde\sigma(0)=0$, it follows that the permutation $\mathrm{end}(\hat\sigma)$ induced by $\hat \sigma$ on the set of ends  $\{0,\infty\}$ is $\id_{\{0,\infty\}}$. Therefore
\begin{equation*}
\begin{split}
\lim_{n\to \infty} (\hat{\sigma}\circ\ f \circ\ \hat{\sigma}^{-1})^n(w_0)&=\lim_{n\to \infty} (\hat{\sigma}\circ\ f^n \circ\ \hat{\sigma}^{-1})(w_0)=\mathrm{end}(\hat\sigma)(\lim_{n\to \infty}f^n(\hat{\sigma}^{-1}(w_0)))\\
&=\mathrm{end}(\hat\sigma)(0)=0,	
\end{split}	
\end{equation*}
whereas $\lim_{n\to \infty} f^{-1}(w_0)=\infty$. Therefore the case $\hat{\sigma}\circ\ f \circ\ \hat{\sigma}^{-1}=f^{-1}$ is ruled out.
\end{proof}
 
We introduce the subclass $II'_c$ of $II_c$ defined by
$$
 II'_c\coloneq   \left\{f:W\to W|\ f\begin{pmatrix}
z\\w 	
\end{pmatrix}=\begin{pmatrix}\alpha z\\ \bar \alpha w\end{pmatrix}\vline\    0<|\alpha|<1,\ \alpha\not\in\R \right\}.
$$

\begin{pr}\label{AH-autom}
 Let $f\in  {IV}\cup  {III}\cup  {II_a}\cup  {II_b}\cup  {II_c}$. The following conditions are equivalent:
\begin{enumerate}
	\item The primary Hopf surface $H_f$ admits anti-holomorphic automorphisms.
	\item Either  the coefficients of $f$ are real, or $f\in II'_c$.
\end{enumerate}	
\end{pr}
\begin{proof}
The  $H_f$ admits an anti-holomorphic automorphism   if and only if $H_f$ is biholomorphic to $\bar H_f$. On the other hand the conjugation automorphism   $c:W\to W$ induces an anti-holomorphic isomorphism $s:H_f\to H_{\fg}$, where $\fg\coloneq c\circ f\circ c^{-1}$. Therefore $\bar H_f\simeq H_{\fg}$, so $H_f$ admits an anti-holomorphic automorphism if and only if $H_f\simeq H_{\fg}$.
 Now note that  $\fg$ is obtained from $f$ by conjugating the coefficients of the polynomial expression which defines $f$. 
On the other hand  Werner's classes are conjugation invariant, in particular  $\fg $ also belongs to ${IV}\cup  {III}\cup  {II_a}\cup  {II_b}\cup  {II_c}$.
By the classification Theorem \ref{Wehl}, it follows that $H_f\simeq H_{\fg}$  if and only if either the coefficients of $f$ and $\fg$ coincide (in other words the coefficients of $f$ are real), or $f$ and $\fg$ belong to $II_c$ and the coefficients $\bar\alpha$, $\bar\delta$ of $\fg$  are obtained from the the coefficients $\alpha$, $\delta$ of $f$ by changing the order. The latter condition is equivalent to $f\in II'_c$.

 \end{proof}

\begin{re}\label{new-direct-proof}
A direct proof of  Proposition \ref{AH-autom} can be obtained	using Proposition \ref{PrHolAnti} and the Taylor expansion of the anti-holomorphic automorphism $\tilde \sigma$ of $\C^2$ obtained by applying Hartogs Theorem to the lift $\hat\sigma$ of an anti-holomorphic automorphism $\sigma$.
\end{re}

For a primary Hopf surface $H$ we  denote by $\mathrm{Ah}(H)$ the set of anti-holomorphic automorphisms. If $H=H_f$ with $f\in  {IV}\cup  {III}\cup  {II_a}\cup  {II_b}\cup  {II_c}$ this set   can be obtained explicitly using the idea of Remark \ref{new-direct-proof}. An anti-holomorphic automorphism $\sigma\in \mathrm{Ah}(H)$ has a lift $\hat \sigma\in \mathrm{Ah}(W)$, which extends to an anti-holomorphic automorphism $\tilde\sigma\in  \mathrm{Ah}(\C^2)$ with $\tilde\sigma (0)=0$. Denoting by $\tilde f\in\Aut_h(\C^2)$ the extension of $f$, we see that the condition $\hat \sigma\circ f=  f\circ \hat \sigma$ is equivalent to $\tilde \sigma\circ\tilde f=\tilde f\circ \tilde \sigma$, which can be interpreted in terms of the Taylor expansion 
$$
\tilde \sigma(z,w):=\left(\sum_{p,q\in \N}a_{pq}\bar z^p\bar w^q,\sum_{p,q\in \N}b_{pq}\bar z^p\bar w^q\right)
$$
of $\tilde \sigma$ at 0. Using this method we obtain easily  
\begin{pr}\label{ahAuto} Let $f\in  {IV}\cup  {III}\cup  {II_a}\cup  {II_b}\cup  {II_c}$ with real coefficients. The set
$$
\mathrm{Ah}(W)^f\coloneq \{u\in \mathrm{Ah}(W)|\ u\circ f=f\circ u\}
$$
is given by the table:
\begin{equation}
\begin{array}{|c|c|}
  \hline &\\ [-0.8em]
\hbox{The class of }f &   \mathrm{Ah}(W)^f	\\ [0.2em]
\hline &\\ [-0.8em]
IV   &\left\{\begin{pmatrix}z\\w\end{pmatrix}\mapsto A\begin{pmatrix}\bar z\\\bar w\end{pmatrix}\vline\ A\in\GL(2,\C)\right\}\\ [0.8em]
\hline  &\\ [-0.8em]
III    &\left\{\begin{pmatrix}z\\w\end{pmatrix}\mapsto \begin{pmatrix}a\bar z+b\bar w^r\\d\bar w\end{pmatrix}\vline\   a\in\C^*,\ d\in\C^*,\ b\in \C\right\}\\ [0.8em]
\hline &\\ [-0.8em]
II_a   &\left\{\begin{pmatrix}z\\w\end{pmatrix}\mapsto \begin{pmatrix}a^r\bar z+b\bar w^r\\a \bar w\end{pmatrix}\vline\   a\in\C^*,\ b\in \C\right\}\\ [0.8em]
\hline &\\ [-0.8em]
II_b   &\left\{\begin{pmatrix}z\\w\end{pmatrix}\mapsto \begin{pmatrix}a \bar z+b\bar w \\a \bar w\end{pmatrix}\vline\   a\in\C^*,\ b\in \C\right\}\\ [0.8em]
\hline &\\ [-0.8em]
II_c\setminus II'_c   &\left\{\begin{pmatrix}z\\w\end{pmatrix}\mapsto \begin{pmatrix}a \bar z  \\ d \bar w\end{pmatrix}\vline\     a\in\C^*,\ d\in\C^*\right\}\\ [0.8em]
\hline  &\\ [-0.8em]
II'_c &\left\{\begin{pmatrix}z\\w\end{pmatrix}\mapsto \begin{pmatrix}a \bar w  \\ d \bar z\end{pmatrix}\vline\     a\in\C^*,\ d\in\C^*\right\}\\
[0.8em]
\hline
\end{array}\ .	
\end{equation}

In each case the cyclic group $\langle f\rangle$ acts freely on the set $\mathrm{Ah}(W)^f$, and $\mathrm{Ah}(H_f)$ is identified with the quotient set $\mathrm{Ah}(W)^f/\langle f\rangle$.

\end{pr}

\section{The classification of Real primary Hopf surfaces}\label{sect2}

\subsection{Even Real structures, odd Real structures on  primary Hopf surfaces}

We start by the following simple result:

\begin{pr}\label{liftsOfsigma}	
	Let $H=H_f$ be a primary Hopf surface. Let $\sigma:H\to\ H$ be an anti-holomorphic involution of $H$ and $\hat{\sigma}:W\to\ W$ be a lift of $\sigma$ (see Proposition \ref{PrHolAnti}). Then 
	
	\begin{enumerate}
		\item There exists $n\in \Z$ such that $\hat{\sigma}^2=f^n$.
	
			\item The parity of   $n$ in the previous formula is well defined (depends only on $\sigma$, not on the lift $\hat{\sigma}$).
	\end{enumerate}

	\end{pr}
	\begin{proof}
	Indeed, since $\sigma^2=\id_{H_f}$ it follows that	$\hat{\sigma}^2\in\Aut_H(W)=\langle f\rangle$. This proves the first claim. 
	
	 Let now  $\hat\sigma$,  $\hat\sigma'$  two lifts of $\sigma$ and $n$, $n'\in\Z$ be the associated integers. We  have $\hat\sigma'\circ \hat\sigma^{-1}\in\Aut_H(W)$, so there exists $k\in\N$ such that $\hat\sigma'=\hat\sigma\circ f^k$.  Since $\hat\sigma$ commutes with $f$, we get $\hat\sigma'^2=\hat\sigma^2\circ f^{2k}$, so $n'=n+2k$.
	\end{proof}

	Taking into account this proposition, it's natural to define:

\begin{dt} \label{even-odd-def} A Real structure $\sigma:H \to\ H$ on $H$ is said to be:	
	\begin{enumerate}
		\item even, if one of the following equivalent conditions is verified:
	\begin{enumerate}
	\item For any lift $\hat{\sigma}:W\to\ W$ of $\sigma$, $\hat{\sigma}^2$ coincides with an even power of $f$.
	\item There exists a lift $\hat{\sigma}:W\to\ W$ of $\sigma$ such that $\hat{\sigma}^2=\id_W$.      
	\end{enumerate}
	
	\item odd, if one of the following equivalent conditions is verified:
	\begin{enumerate}
	\item For any lift $\hat{\sigma}:W\to\ W$ of $\sigma$, $\hat{\sigma}^2$ coincides with an odd power of $f$.
	\item There exists a lift $\hat{\sigma}:W\to\ W$ of $\sigma$ such that $\hat{\sigma}^2=f$.
	\end{enumerate}
	\end{enumerate}	
	\end{dt}

\begin{ex}\label{ExStandard}
Let $f\in IV\cup III\cup II_a\cup II_b\cup II_c$ be	with real coefficients. The standard complex conjugation $c:W\to W$ induces a Real structure $s_f$ on $H_f$, which is obviously even. The Real structure will be called {\it the standard Real structure} of $H_f$.

Let now $f\in II'_c$ be given by $f\begin{pmatrix}
z\\w	
\end{pmatrix}=\begin{pmatrix}
\alpha z\\\bar \alpha w	
\end{pmatrix}$,	 where $0<|\alpha|<1$ and $\alpha\not\in\R$. The anti-holomorphic automorphism of $c':W\to W$ defined by
$$
 c'\begin{pmatrix}
z\\w	
\end{pmatrix}= \begin{pmatrix}
\bar w\\\bar z	
\end{pmatrix}
$$
 commutes with $f$ and is involutive, so it defines an even Real structure $s_f$ on $H_f$, which will also be   called {\it the standard Real structure} on $H_f$.
\end{ex}

\subsection{The classification of even Real structures on primary Hopf surfaces}

Let $E$ be a complex vector space of dimension $n$. Recall that a real structure on $E$ is an anti-linear involution $a:E\to E$. Recall that any two real structures $a$, $b:E\to E$ on $E$ are equivalent, i.e. there exists $l\in\GL(E)$ such that $a=l\circ b\circ l^{-1}$. Indeed, putting
$$
E^a_\pm\coloneq \ker( a\mp \id_E), \ E^b_\pm\coloneq \ker( b\mp \id_E)
$$
we have real direct sum decompositions
$$
E=E^a_+\oplus E^a_-,\ E=E^b_+\oplus E^b_-
$$
with $E^a_-=i E^a_+$, $E^b_-=i E^b_+$. Choose an   $\R$-linear isomorphism $h:E^b_+\to E^a_+$ and note that $l:E\to E$ defined by
$$
l(x+iy)\coloneq h(x)+i h(y) \hbox{ for any } x,\ y\in E^b_+
$$
is a $\C$-linear isomorphism satisfying  $a\circ l=l\circ b$. In particular, in the special case when $E=\C^n$ and $b$ is the standard conjugation $c:\C^n\to\C^n$, we obtain
 
\begin{re}\label{linearRealRem}
Let $a:\C^n\to\C^n$ be an anti-linear involution on $\C^n$. Then there exists a $\C$-linear automorphism $l:\C^n\to \C^n$ such that $a=l\circ c\circ l^{-1}$. 

In other words  any anti-linear Real structure on $\C^n$ is $\GL(n,\C)$-conjugate to the standard conjugation.  
\end{re}

The   classification of even Real structures on primary Hopf surfaces will follow from the following proposition:
\begin{pr}\label{PropEven}
Let $f\in IV\cup III\cup II_a\cup II_b\cup II_c$ be	with real coefficients and let $\phi$ be a Real structure on $W$ such that $\phi\circ f=f\circ \phi$. Then there exists $\psi\in \Aut_h(W)^f$ such that $\phi=\psi\circ c\circ \psi^{-1}$.

\end{pr}
\begin{proof}
\begin{enumerate}
\item $f\in IV$.  In this case, using Proposition \ref{ahAuto} we see that $\phi$ has the form
$$
\phi\begin{pmatrix}
z\\ w	
\end{pmatrix}=\begin{pmatrix}
a & b\\ c & d	
\end{pmatrix}\begin{pmatrix}
\bar z\\  \bar w	
\end{pmatrix}
$$ 
with $\begin{pmatrix}
a & b\\ c & d	
\end{pmatrix}\in\GL(2,\C)$.  The obvious extension   $\tilde\phi:\C^2\to\C^2$ of $\phi$ to $\C^2$ is an anti-linear Real structure on $\C^2$, so   Remark \ref{linearRealRem} applies and gives a $\C$-linear isomorphism $l:\C^2\to\C^2$  such that $\tilde \phi=l\circ c\circ l^{-1}$. Denoting by $\psi:W\to W$ the automorphism induced by $l$ (which obviously commutes with $f$), we get $\phi=\psi\circ c\circ \psi^{-1}$ as claimed. \vspace{2mm}

%
\item $f\in III$.   In this case  Proposition \ref{ahAuto} shows that $\phi$ is given by 
$$\phi\begin{pmatrix}z\\w\end{pmatrix}= \begin{pmatrix}a\bar z+b\bar w^r\\d\bar w\end{pmatrix},
$$
where $a$, $d\in\C^*$, $b\in \C$. We have
$$
\phi^2\begin{pmatrix}z\\w\end{pmatrix}=\begin{pmatrix}\vert a\vert^2z+(a\bar b+b\bar d^r)w^r\\\vert d\vert^2w\end{pmatrix}.
$$
The condition $\phi^2=\id$ becomes
\begin{equation}\label{phi2=idIII}
\left\{\begin{array}{ccc}
\vert a\vert^2&=&1 \\
a\bar b+b\bar d^r&=&0 \\
\vert d\vert^2&=&1 
 \end{array}\right..
 \end{equation}
 Let now $ \psi\in\Aut_h(W)^f$. By Wehler's Theorem \ref{Wehl} we have
 $$
  \psi\begin{pmatrix}z\\w\end{pmatrix}=\begin{pmatrix} Az+Bw^r\\ Dw\end{pmatrix}
 $$
 with $A$, $D\in\C^*$, $B\in\C$. 	The condition
$$\phi=\psi\circ\ c \circ\ \psi^{-1}$$ 
is equivalent to the system
\begin{equation}\label{eqIII}
\begin{cases}
a=A\bar A^{-1}  \\
b=B\bar  D^{-r}-A\bar A^{-1} \bar B\bar D^{-r} \\
d=D\bar D^{-1}. 
\end{cases}
\end{equation}
Since $|a|=|d|=1$ we can write $a=e^{i\theta}$, $d=e^{i\delta}$ with $\theta$, $\delta\in\R$. Put $A\coloneq e^{i\frac{\theta}{2}}$, $D\coloneq e^{i\frac{\delta}{2}}$. With this choice we have $A\bar A^{-1}=A^2=a$, $D\bar D^{-1}=D^2=d$, so the first and the third equations in (\ref{eqIII}) are satisfied.  We are seeking $B\in\C$ such that the second equation is also satisfied.

Consider the $\R$-linear map $\lambda:\C\to\C$  defined by
$$
\lambda(z)=u z-v\bar z,
$$
where $u\coloneq \bar  D^{-r}$,  $v\coloneq A\bar A^{-1}  \bar D^{-r}= A^2\bar D^{-r}$. This map is not surjective. Using $|u|=|v|=1$  it follows easily that its image is the real line
\begin{equation}\label{imlambda}
\im(\lambda)=\{\zeta\in\C|\ u^{-1}\zeta+v \bar \zeta=0\}\subset\C.
\end{equation}
Now note that $b\in\im(\lambda)$, because
$$
\hspace*{6mm}u^{-1}b+ v\bar b=\bar  D^{r}b+A^2\bar D^{-r}\bar b=\bar D^{-r}(\bar D^{2r}b+   A^2\bar b)=\bar D^{-r}(\bar d^rb+a\bar b)=0
$$
by the second equation in (\ref{phi2=idIII}). Therefore, there exists $B\in\C$ such that
$$
b=\lambda(B)=\bar  D^{-r} B-A\bar A^{-1}\bar D^{-r} \bar B,
$$
which proves that, with this choice, the second equation in (\ref{eqIII}) is  satisfied, too.
\vspace{2mm}

\item $f\in II_a$. In this case Proposition \ref{ahAuto} shows that $\phi$ has the form 
$$\phi\begin{pmatrix} z\\w\end{pmatrix}=\begin{pmatrix}a^r\bar z+b\bar w^r\\a\bar w\end{pmatrix}$$
 with $a\in\C^*$ and $b\in \C$. The condition $\phi^2=\id_W $ is equivalent to the system
\begin{equation}\label{phi2=idIIa}
\left\{\begin{array}{ccc}
\vert a\vert^2&=&1 \\
a^r\bar b+b\bar a^r&=&0 		
\end{array}\right..
\end{equation}
By Wehler's theorem, an automorphism $\psi\in \Aut_h(W)^f$ has the form
$$\psi\begin{pmatrix}z\\w\end{pmatrix}=\begin{pmatrix}A^rz+Bw^r\\Aw\end{pmatrix}$$
with $A\in\C^*$, $B\in \C$. The condition
$$\phi=\psi\circ\ c \circ\ \psi^{-1}$$ 
is equivalent to the system 
\begin{equation}\label{eqIIa}
\begin{cases}
a=A\bar A^{-1}  \\
b=B\bar A^{-r}-A^{r} \bar B\bar A^{-2r}. 
\end{cases}
\end{equation}
Since $|a| ^2=1$  we can write
	 $a=e^{i\theta}$ with $\theta\in\R$. Putting $A=e^{\frac{i\theta}{2}}$ we have $A\bar A^{-1}=A^2=a$. As in the previous case consider the  $\R$-linear map $\lambda:\C\to  \C$   
	$$\lambda(z)=uz-v\bar z,$$ 
where this time  we choose $u\coloneq A^r$ and $v\coloneq A^{3r}$. We have again $|u|=|v|=1$, so  the image of $ \lambda$ is again given by (\ref{imlambda}). Note that $b\in\im(\lambda)$ because
$$
u^{-1}b+ v \bar b=A^{-r} b+A^{3r} \bar b=A^r(A^{2r}\bar b+ A^{-2r} b)=A^r(a^r \bar b+ b\bar a^r)=0
$$
by the second equation in (\ref{phi2=idIIa}). Therefore there exists $B\in\C$ such that
$$
b=\lambda(B)= A^r B- A^{3r} \bar B=B\bar A^{-r}- A^r \bar B\bar A^{-2r},
$$
which shows that $(A,B)\in\C^*\times\C$ is a solution of the system (\ref{eqIIa}). 
\vspace{2mm}

\item $f\in II_b$.   In this case Proposition \ref{ahAuto} shows that $\phi$ has the form 
$$\phi\begin{pmatrix}
z\\w	
\end{pmatrix}
=\begin{pmatrix} a\bar z+b\bar w\\a\bar w\end{pmatrix}$$
 where $a\in \C^*$, $b\in\C$.  The condition $\phi^2=\id_W $ is equivalent to
\begin{equation}\label{phi2=idIIb}
\left\{\begin{array}{ccc}
\vert a\vert^2&=&1 \\
a\bar b+b\bar a&=&0
\end{array}\right..
\end{equation}

By Wehler's theorem, an automorphism $\psi\in \Aut_h(W)^f$  has the form 
$$\psi\begin{pmatrix}
z\\w	
\end{pmatrix}=\begin{pmatrix}
Az+Bw\\Aw	
\end{pmatrix}$$
where $A\in \C^*$, $B\in\C$. We use the same arguments as in the case $II_a$ but taking $r=1$ in all formulae used in this case. 
\vspace{2mm}

\item $f\in II_c$. In this case Proposition \ref{ahAuto} shows that $\phi$ has the form 
$$\phi\begin{pmatrix}
z\\w	
\end{pmatrix}
=\begin{pmatrix} a\bar z \\ b\bar w\end{pmatrix}$$
 where $a$, $b\in \C^*$.  The condition $\phi^2=\id_W $ is equivalent to
\begin{equation}\label{phi2=idIIc}
\left\{\begin{array}{ccc}
| a|^2&=&1 \\
|b|^2&=&1
\end{array}\right..
\end{equation}

By Wehler's theorem, in this case an automorphism $\psi\in \Aut_h(W)^f$   has the form 
$$\psi\begin{pmatrix}
z\\w	
\end{pmatrix}=\begin{pmatrix}
Az \\Bw	
\end{pmatrix}$$
where $A$, $B\in \C^*$, and the condition 
$$\phi=\psi\circ\ c \circ\ \psi^{-1}$$
is equivalent to the system
\begin{equation}\label{eqIIc}
\begin{cases}
a=A\bar  A^{-1}  \\
b=B\bar B^{-1}. 
\end{cases}	
\end{equation}
It suffices to put $A\coloneq e^{i\frac{\theta}{2}}$, $B\coloneq e^{i\frac{\beta}{2}}$, where $a=e^{i \theta}$, $b=e^{i \beta}$.
\end{enumerate}
\end{proof}

Using the notation $c'$   introduced in Example \ref{ExStandard} we can state the following analogue of Proposition \ref{PropEven} for the the subclass $II'_c$:

\begin{pr}\label{PropEvenII'c}
Let  $f\in II'_c$ be given by $f\begin{pmatrix}
z\\w	
\end{pmatrix}=\begin{pmatrix}
\alpha z\\\bar \alpha w	
\end{pmatrix}$,	 where $0<|\alpha|<1$ and $\alpha\not\in\R$. Let $\phi$ be a Real structure on $W$ such that $\phi\circ f=f\circ \phi$. There exists $\psi\in\Aut_h(W)^f$ such that $\phi=\psi\circ c'\circ \psi^{-1}$.
\end{pr}
\begin{proof}
Using Proposition \ref{ahAuto} we see that $\phi$ has the form
$
\phi\begin{pmatrix}
 z\\  w	
\end{pmatrix}=\begin{pmatrix}
a \bar w\\ d \bar z	
\end{pmatrix}$,
and the condition $\phi^2=\id_W$ becomes 	$\bar a= d^{-1}$.  It suffices to note that
$$
 \begin{pmatrix}
 z\\  w	
\end{pmatrix}\mapsto \begin{pmatrix}
 a z\\    w	
\end{pmatrix}
$$
defines an element $\psi\in \Aut_h(W)^f$ and that $\phi=\psi\circ c'\circ \psi^{-1}$.
 \end{proof}

Using Propositions \ref{PropEven}, \ref{PropEvenII'c} we obtain the following classification theorem for even Real structures on primary Hopf surfaces:
\begin{thry}\label{ClassEven}
Let $f\in IV\cup III\cup II_a\cup II_b\cup II_c$  be either	with real coefficients or element of the subclass $II'_c$. Let $\sigma:H_f\to H_f$ be an even Real structure on $H_f$. There exists a holomorphic automorphism $g\in \Aut_h(H_f)$ such that $\sigma=g\circ s_f\circ g^{-1}$, where $s_f$ is the standard Real structure on $H_f$. 	
\end{thry}
In other words, any {\it even} Real structure on $H_f$ is equivalent to its standard Real structure $s_f$.
\begin{proof}
Since $\sigma$ is even, there exists a lift $\hat \sigma:W\to W$ which is an involution, hence  a Real structure on $W$. By Propositions  \ref{PropEven}, \ref{PropEvenII'c} there exists $\psi\in\Aut_h(W)^f$ such that $\hat\sigma=\psi\circ c\circ \psi^{-1}$, respectively $\hat\sigma=\psi\circ c'\circ \psi^{-1}$. Denoting by $g\in\Aut_h(H_f)$ the automorphism induced by $\psi$, we obtain $\sigma= g\circ s_f\circ g^{-1}$.	
\end{proof}

Note that Proposition \ref{liftsOfsigma} and Definition \ref{even-odd-def} generalize in a natural way to primary Hopf $n$-folds for any $n\geq 2$. Moreover, for a primary  Hopf $n$-fold $H_f$ defined by a holomorphic contraction $f$ given by a polynomial formula with real coefficients, we can define the standard Real structure  $s_f$ of $H_f$ as in Example \ref{ExStandard}.

Using   Remark \ref{linearRealRem} we obtain in the same way the classification of even Real structures on a primary Hopf $n$-fold 
$$H_{f_\alpha}\coloneq \qmod{\C^n\setminus\{0\}}{\langle f_\alpha\rangle},$$
 where $0<|\alpha|<1$ and $f_\alpha(z)=\alpha z$.

\begin{re} 
The primary Hopf $n$-fold $H_{f_\alpha}$ admits an even Real structure if and only if $\alpha\in\R$. If this is the case, any even Real structure on $H_{f_\alpha}$ is equivalent to its standard Real structure $s_{f_\alpha}$. 	
\end{re}

\subsection{The classification of odd Real structures on primary Hopf surfaces}

The classification of odd Real  structures is more difficult.  We consider first the case when the diagonal coefficients of $f$ (denoted by $ \alpha$, $\delta$, $\delta^r$ in Theorem \ref{Wehl}) are positive.  We will start with the following  remark which can be proved easily by direct computations:
\begin{re}\label{kroot}
Let $f\in IV \cup III\cup  II_a\cup II_b\cup II_c$	be with real coefficients and positive diagonal coefficients, and let $k\in\N^*$. The automorphism $f^{\frac{1}{k}}\in\Aut_h(W)$ defined in the table below  has the properties:
\begin{enumerate}
\item $f^{\frac{1}{k}}$ is a polynomial holomorphic contraction with real coefficients. 
\item $f^{\frac{1}{k}}$ is a root of order $k$ of $f$.
\item \label{Auf-k-root} $\Aut_h(W)^{f^{\frac{1}{k}}}=\Aut_h(W)^{f}$, $\Ah(W)^{f^{\frac{1}{k}}}=\Ah(W)^{f}$.
\end{enumerate}

$$  \begin{array}  {|c|c|c|}
\hline &&\\ [-1em]
\hbox{The class of } f &  f\begin{pmatrix}
z\\w 	
\end{pmatrix}& f^{\frac{1}{k}}\begin{pmatrix}
z\\w 	
\end{pmatrix}   
 \\ && \\ [-1em]
 \hline &&\\ [-1em]
IV & \begin{pmatrix}\alpha z\\ \alpha w\end{pmatrix} & \begin{pmatrix}\alpha^{\frac{1}{k}} z\\ \alpha^{\frac{1}{k}} w\end{pmatrix}
\\ && \\ [-1em]
 \hline &&\\ [-1em]
III & \begin{pmatrix}\delta^r z\\ \delta w\end{pmatrix}  & \begin{pmatrix}
 \delta^{\frac{r}{k}}z\\ \delta^{\frac{1}{k}}w 	
\end{pmatrix}   
\\ && \\ [-1em]
 \hline &&\\ [-1em]
II_a  & \begin{pmatrix}\delta^r z+w^r\\ \delta w\end{pmatrix} &  \begin{pmatrix}
 \delta^{\frac{r}{k}}z+\frac{1}{k}\delta^{r\big(\frac{1-k}{k}\big)}w^r
\\ 
\delta^{\frac{1}{k}}w	
\end{pmatrix}  
\\ && \\ [-1em]
 \hline &&\\ [-0.7em]
II_b & \begin{pmatrix}\alpha z+w\\ \alpha w\end{pmatrix} & \begin{pmatrix}
 \alpha^{\frac{1}{k}}z+\frac{1}{k}\alpha^{\frac{1-k}{k}}w
\\
\alpha^{\frac{1}{k}}w	
\end{pmatrix}
\\ && \\ [-1em]
\hline &&\\ [-1em]
II_c &\begin{pmatrix}\alpha z\\ \delta w\end{pmatrix} & \begin{pmatrix}
 \alpha^{\frac{1}{k}}z
\\
\delta^{\frac{1}{k}}w	
\end{pmatrix}  
\\ [0.8em] 
\hline
\end{array}
$$
\end{re}
\begin{pr} \label{class-odd-pos-coeff}
Let $f\in IV \cup III\cup  II_a\cup II_b\cup II_c$	be with real coefficients and positive diagonal coefficients, and let $f^{\frac{1}{2}}$ be  the square root of $f$ given by Remark \ref{kroot}. Let $c:W\to W$ be the standard conjugation. Then:
\begin{enumerate}
\item \label{ccircf1/2} The composition $c\circ f^{\frac{1}{2}}$ belongs to $\mathrm{Ah}(W)^f$ and satisfies  $(c\circ f^{\frac{1}{2}})^2=f$.
\item For any $\phi\in \mathrm{Ah}(W)^f$ with $\phi^2=f$ there exists $\psi\in\Aut_h(W)^f$  such that
$$
\phi=\psi\circ (c\circ f^{\frac{1}{2}})\circ \psi^{-1}.
$$
\end{enumerate}	
\end{pr}

In other words, under the assumption of the proposition, any odd Real structure on $H_f$ is equivalent to the odd real structure $\sigma_f$ induced by $c\circ f^{\frac{1}{2}}$.
\begin{proof}
(1) This follows taking into account that, since $f^{\frac{1}{2}}$ has real coefficients, it commutes with $c$. 
\vspace{2mm}\\
(2)
Put $\phi'=f^{-\frac{1}{2}}\circ\phi$, and note that $\phi$ belongs to $\mathrm{Ah}(W)^f$ and (since $\phi$ commutes with $f^{-\frac{1}{2}}$ by Remark \ref{kroot} (\ref{Auf-k-root}))  satisfies $\phi'^2=\id_W$. Therefore $\phi'$ is a Real structure on $W$ which commutes with $f$. By Proposition \ref{PropEven}, there exists $\psi\in \Aut_h(W)^f$ such that $\phi'=\psi\circ c\circ \psi^{-1}$. Therefore
$$f^{-\frac{1}{2}}\circ\phi=\psi\circ c\circ \psi^{-1}.
$$
By   Remark \ref{kroot} (\ref{Auf-k-root}) again $\psi$ commutes with $f^{\frac{1}{2}}$, so we obtain
$$
\phi= \psi\circ (c\circ f^{\frac{1}{2}})\circ \psi^{-1},
$$
as claimed. 
\end{proof}

The proposition below shows that if $H_f$   admits an odd Real structure and 
$$f\in   III\cup  II_a\cup II_b\cup II_c,$$
 has real coefficients, then its diagonal coefficients are always positive, so Proposition \ref{class-odd-pos-coeff} classifies  odd Real structure on all primary Hopf surfaces except those of class $II'_c$ and those of class $IV$ with negative parameter $\alpha$.
 \begin{pr}
 Let $f\in   III\cup  II_a\cup II_b\cup II_c$ be with real coefficients.	 The following conditions are equivalent:
\begin{enumerate}
\item There exists $\phi\in \mathrm{Ah}(W)^f$ such that $\phi^2=f$.
\item The diagonal coefficients of  $f$ are positive.
\end{enumerate}
 \end{pr}
\begin{proof} $(1)\Rightarrow(2)$:
Suppose $f\in III$, so $f\begin{pmatrix} z\\  w\end{pmatrix}=\begin{pmatrix}\delta^r z\\ \delta w\end{pmatrix}$	with $0<|\delta|<1$. An element $\phi\in \mathrm{Ah}(W)^f$ has the form 
$$
\phi\begin{pmatrix} z\\  w\end{pmatrix}= \begin{pmatrix}a\bar z+b\bar w^r\\d\bar w\end{pmatrix}.
$$
The condition $\phi^2=f$ becomes
$$\left\{\begin{array}{ccc}
\vert a\vert^2&=&\delta^r \\
a\bar b+b\bar d^r&=&0 \\
\vert d\vert^2&=&\delta 
\end{array}\right.,$$
which obviously implies $\delta>0$. The other cases are treated in a similar way.
\vspace{2mm}\\
$(2)\Rightarrow(1)$: This follows from Proposition \ref{class-odd-pos-coeff} (\ref{ccircf1/2}).
\end{proof}

The analogue of Proposition \ref{class-odd-pos-coeff} for $f\in II'_c$ follows easily by direct computation: 
\begin{pr}\label{class-odd-II'c} 
Let  $f\in II'_c$ be given by $f\begin{pmatrix}
z\\w	
\end{pmatrix}=\begin{pmatrix}
\alpha z\\\bar \alpha w	
\end{pmatrix}$,	 where $0<|\alpha|<1$ and $\alpha\not\in\R$.	 Let $\alpha^{\frac{1}{2}}$ be a square root of $\alpha$, $\bar \alpha^{\frac{1}{2}}$ be its conjugate, and let $f^{\frac{1}{2}}$ be the square root of $f$ defined by $f^{\frac{1}{2}}\begin{pmatrix}
z\\w	
\end{pmatrix}=\begin{pmatrix}
\alpha^{\frac{1}{2}} z\\ \bar \alpha^{\frac{1}{2}} w	
\end{pmatrix}$. Then  
\begin{enumerate}
\item $c'$ commutes with $f^{\frac{1}{2}}$.
\item $c'\circ 	f^{\frac{1}{2}}\in \Ah(W)^f$ satisfies $(c'\circ 	f^{\frac{1}{2}})^2=f$.
\item For any $\phi\in  \Ah(W)^f$ with $\phi^2=f$ there exists $\psi\in\Aut_h(W)^f$ such that $\phi=\psi\circ(c'\circ 	f^{\frac{1}{2}})\circ\psi^{-1}$.
\end{enumerate}
\end{pr}

In other words, under the assumption of the proposition, for $f\in II'_c$, any odd Real structure on $H_f$ is equivalent to the odd real structure $\sigma_f$ induced by $c'\circ f^{\frac{1}{2}}$.\\

For the classification of odd   Real  structures on class IV primary Hopf surfaces with negative parameter $\alpha$ we will need a simple remark concerning the classification of quaternionic structures on  a finite dimensional real vector space. 

Let $F$ be a real vector space of dimension $4k$. Recall that a left $\H$-vector space structure on $F$ is equivalent to the data of a pair  $(I,J)\in\End(F)\times\End(F)$ such that $I^2=J^2=-\id_F$ and $I\circ J=-J\circ I$. In the presence of such a pair $(I,J)$, we put $K\coloneq I\circ J$, and we define a left quaternonic vector space structure on $F$ (extending its real space structure) by mapping  the quaternonic units $i$, $j$, $k\in\H$  to $I$, $J$, $K$ respectively.

Since two left $\H$-vector spaces of the same dimension are isomorphic, it follows that any two left $\H$-vector space structures $(I,J)$,  $(I',J')$ on $F$ are always equivalent, i.e. there exists an automorphism $l\in\GL(F)$ such that $l\circ I=I'\circ l$,  $l\circ J=J'\circ l$. In the special case $I=I'$ we obtain $l\circ I=I\circ l$, i.e. $l$ is linear with respect to the complex structure defined by $I$, and $J'=l\circ J\circ l^{-1}$. This shows that 
\begin{re}\label{RemQuat}
Let $E$ be a complex vector space of dimension $2k$ and let $J$, $J'$ be anti-linear isomorphisms such that $J^2=J'^2=-\id_E$.  Then there exists an automorphism $l\in\GL(E)$ such that $J'=l\circ J\circ l^{-1}$.
\end{re}

Consider the  anti-linear map $J:\C^{2k} \to \C^{2k}$ defined by 
$$J\begin{pmatrix} z\\w\end{pmatrix} =\begin{pmatrix}-\bar w\\ \bar  z\end{pmatrix}=\begin{pmatrix} 0 & -\id_{\C^k} \\ \id_{\C^k}& 0 \end{pmatrix}\begin{pmatrix} \overline z\\   \overline w \end{pmatrix}$$
for pairs $(z,w)\in \C^k\times\C^k=\C^{2k}$. Note that $J^2=-\id_{\C^{2k}}$.
By Remark \ref{RemQuat} we obtain the following analogue of Remark \ref{linearRealRem}:
\begin{re}\label{ClassQuatStr} Let $a:\C^{2k}\to\C^{2k}$  ba an anti-linear isomorphism  satisfying $a^2=-\id_{\C^2}$. Then there exists $l\in\GL(2k,\C)$ such that $a=l\circ J\circ l^{-1}$, i.e.   $a$ is $\GL(2k,\C)$-conjugate to $J$.
	
\end{re}
With this preparation we can prove:
\begin{pr}\label{oddIV-neg-coeff}
	Let $f:W\to\ W\in IV$ be given by $f\begin{pmatrix}z\\w\end{pmatrix}=\begin{pmatrix}\alpha z\\\alpha w\end{pmatrix}$,
where $\alpha\in(-1,0)$. Let $\alpha^{\frac{1}{2}}$ be a square root of $\alpha$, and let $f^{\frac{1}{2}}$ be the square root of $f$ defined by $f^{\frac{1}{2}}\begin{pmatrix}z\\w\end{pmatrix}=\begin{pmatrix}\alpha^{\frac{1}{2}} z\\\alpha^{\frac{1}{2}} w\end{pmatrix}$. Then
\begin{enumerate}
\item $J$ anti-commutes with $f^{\frac{1}{2}}$.
\item $J\circ f^{\frac{1}{2}}\in \Ah(W)^f$ and satisfies	 $(J\circ f^{\frac{1}{2}})^2=f$.
\item For any $\phi\in \Ah(W)^f$ with $\phi^2=f$ there exists $\psi\in \Aut_h(W)$ such that $\phi=\psi\circ(J\circ f^{\frac{1}{2}})\circ\psi^{-1}$.
\end{enumerate}

\end{pr}

Therefore, under the assumptions of the proposition, any odd Real structure $\sigma$ on $H_f$ is  equivalent to $\sg_f$, where $\sg_f$ is the odd Real structure induced by $J\circ f^{\frac{1}{2}}$.

\begin{proof}
(1) It suffices to note that $f^{\frac{1}{2}}=\alpha^{\frac{1}{2}}\id_W$  with $\alpha^{\frac{1}{2}}$ pure imaginary and to recall that $J$ is anti-linear.
\vspace{2mm}\\
(2) Follows from (1) taking into account that $J^2=-\id_{\C^2}$.
\vspace{2mm}\\
(3) Let $\phi\in \Ah(W)^f$ with $\phi^2=f$. The extension $\tilde \phi $ of $\phi$ to $\C^2$ gives an anti-linear isomorphism $\tilde\phi:\C^2\to  \C^2$. Since $\tilde\phi$ anti-commutes with $f^{\pm\frac{1}{2}}$,  the condition  $\phi^2=f$ is equivalent to $(\tilde\phi\circ f^{-\frac{1}{2}})^2=-  \id_{\C^2}$.    By  Remark  \ref{ClassQuatStr}, there exists     $l\in\GL(2,\C)$  such that 
$$\tilde\phi\circ f^{-\frac{1}{2}}=l\circ J\circ  l^{-1}.$$
Denoting by $\psi$ the automorphism of $W$ induced by $l$, which obviously commutes with $f^{\frac{1}{2}}$, we obtain $\phi=\psi\circ (J\circ f^{\frac{1}{2}})\circ  \psi^{-1}$, as claimed. 	
\end{proof}

Using Propositions \ref{class-odd-pos-coeff}, \ref{oddIV-neg-coeff} we obtain the following classification theorem for odd Real structures on primary Hopf surfaces:
\begin{thry}\label{ClassOdd}   
\begin{enumerate}
\item Let $f\in III\cup  II_a\cup II_b\cup II_c$. 
\begin{enumerate} 
\item 	The following conditions are equivalent:
\begin{enumerate}
\item $H_f$ admits an odd Real structure.
\item  $f$  either has real coefficients and positive diagonal coefficients, or belongs to $II'_c$.
\end{enumerate}	
\item If one of these equivalent conditions is satisfied, any odd Real structure on $H_f$ is equivalent to the  Real structure $\sigma_f$ defined above.
\end{enumerate}
\item Let $f\in IV$  be given by $f(z,w)=(\alpha z,\alpha w)$ where $0<|\alpha|<1$. 	
\begin{enumerate}
\item  The following conditions are equivalent:
\begin{enumerate}
\item $H_f$ admits an odd Real structure.
\item  $\alpha\in\R$.
\end{enumerate}	
\item If $\alpha\in (0,1)$, any odd Real structure on $H_f$ is equivalent to  $\sigma_f$. If $\alpha\in (-1,0)$, any odd Real structure on $H_f$ is equivalent to  $\sg_f$.
\end{enumerate}

\end{enumerate}
	
\end{thry}

Consider again  the primary   Hopf  $n$-fold 
$H_{f_\alpha}\coloneq  {\C^{n}\setminus\{0\}}/{\langle f_\alpha\rangle}$.
Using  Remark \ref{ClassQuatStr} and defining  the odd Real structures $\sigma_{f_\alpha}$ (for $\alpha \in(0,1)$), $\sg_{f_\alpha}$  (for $\alpha\in (-1,0)$ and $n$ even) as above, we obtain in a similar way the classification of odd Real structures on  $H_{f_\alpha}$:
\begin{re} Let $H_{f_\alpha}$ be the primary Hopf $n$-fold  associated with the holomorphic contraction $f_\alpha$, where $0<|\alpha|<1$. 
\begin{enumerate}
\item Suppose $n$ is odd. 	$H_{f_\alpha}$ admits an odd Real structure if and only if $\alpha\in(0,1)$.  If this is the case, any odd Real structure on  $H_{f_\alpha}$ is equivalent  to $\sigma_{f_\alpha}$.
\item Suppose $n$ is even. $H_{f_\alpha}$ admits an odd Real structure if and only if $\alpha\in\R$.  If $\alpha\in (0,1)$ any odd Real structure on  $H_{f_\alpha}$ is equivalent  to $\sigma_{f_\alpha}$. If $\alpha\in (-1,0)$ any odd Real structure on  $H_{f_\alpha}$ is equivalent  to $\sg_{f_\alpha}$.
\end{enumerate}
 	 
\end{re}

\section{The automorphism group and the Real Picard group}
\subsection{The automorphism group}
 
 For the automorphism group of an even Real primary Hopf surface $(H_f,s_f)$ we have the following result which follows easily from Theorem \ref{Wehl}:
 \begin{thry}\label{AutEven}
Let $f\in  {IV}\cup  {III}\cup  {II_a}\cup  {II_b}\cup  {II_c}$. 
\begin{enumerate}
\item  Suppose is with real coefficients. The group  $\Aut_h(W)^{f,c}$ of holomorphic automorphisms of $W$  commuting with $f$ and the standard conjugation $c$ is given by the table below:

$$  \begin{array}  {|c|c|}
\hline   & \\ [-0.8em]
\hbox{The class of } f &  \Aut_h(W)^{f,c}   
 \\ [0.1em]
 \hline  & \\ [-0.8em]
IV & \GL(2,\R)  
\\ [0.2em]
 \hline  & \\ [-0.8em]
III &  \left\{\begin{pmatrix}
z\\w 	
\end{pmatrix}\mapsto\ \begin{pmatrix}az+bw^r\\dw\end{pmatrix}\vline\  a\in\R^*,\ d\in\R^*, b\in \R\right\}    
\\ [0.8em]
 \hline  & \\ [-0.8em]
II_a  &\left\{\begin{pmatrix}
z\\w 	
\end{pmatrix}\mapsto\ \begin{pmatrix}a^rz+bw^r\\aw\end{pmatrix}\vline\  a\in \R^*, b\in \R\right\}     
\\ [0.8em]
 \hline  & \\ [-0.8em]
II_b &  \left\{\begin{pmatrix}
z\\w 	
\end{pmatrix}\mapsto\ \begin{pmatrix}az+bw\\aw\end{pmatrix}\vline\  a\in \R^*, b\in  \R\right\}   
\\ [0.8em]
\hline  & \\ [-0.8em]
II_c & \left\{\begin{pmatrix}
z\\w 	
\end{pmatrix}\mapsto\ \begin{pmatrix}az\\dw\end{pmatrix}\vline\  a\in\R^*,\ d\in\R^*\right\}   
\\ [0.8em]
\hline
\end{array}
$$
\item Suppose $f\in II'_c$. Then
$$
\Aut_h(W)^{f,c'}=\left\{\begin{pmatrix}
z\\w 	
\end{pmatrix}\mapsto\ \begin{pmatrix}az\\ \bar a w\end{pmatrix}\vline\  a\in\C^*\right\}.  
$$ 
\item In each case the cyclic group $\langle f\rangle$ is a central subgroup of $\Aut_h(W)^{f,c}$, respectively $\Aut_h(W)^{f,c'}$, and the automorphism group $\Aut_h(H_f,s_f)$ is identified with the quotient  $\Aut_h(W)^{f,c}/\langle f\rangle$, respectively $\Aut_h(W)^{f,c'}/\langle f\rangle$.
\end{enumerate} 
\end{thry}
\begin{proof}
The claims (1), (2) follow directly from Theorem \ref{Wehl}. For (3) it suffices to prove that a holomorphic automorphism $\varphi\in\Aut_h(H_f)$ induced by $\hat \varphi\in 	\Aut_h(W)^f$ commutes with $s_f$ if and only if $\hat\varphi$ commutes with $c$, respectively $c'$. In other words we have to prove that 
$c\circ \hat \varphi= \hat \varphi\circ c\circ f^k$, respectively $c'\circ \hat \varphi= \hat \varphi\circ c'\circ f^k$ (with $k\in\Z)$, then $k=0$. 
This follows by elementary computations. 
\end{proof}

Using Theorem \ref{AutEven} we can describe the automorphism groups of even Real primary Hopf surfaces in terms of (semi-direct products of) classical groups. 
For instance, for $f\in IV$, we obtain $\Aut(H_f,s_f)=\GL(2,\R)/\langle \alpha I_2\rangle$. \vspace{1mm}

For $f\in III$ the group $\Aut_h(W)^{f,c}$ coincides with the group 
$$ 
\big\{ g_{a,d,b}|\  (a,d,b)\in\R^*\times\R^*\times\R\big\},\hbox{ where }   g_{a,d,b}\begin{pmatrix}z\\w\end{pmatrix}=\begin{pmatrix}az+abw^r\\dw\end{pmatrix}.
$$
This group can be identified with the semi-direct product $(\R^*\times\R^*)\ltimes_{\rho_r} \R $ associated with the morphism $\rho_r:(\R^*\times\R^*)\to \GL(1,\R)$  given by 
$$
\rho_r(a,d)(b)=ad^{-r}b. 
$$
Since $\rho_r(\delta^r,\delta)=\id_\R$, it follows that $\rho_r$ descends to a morphism 
$$\hat \rho_r: \qmod{\R^*\times\R^*}{\langle (\delta^r,\delta)\rangle}\to \GL(1,\R).$$

With this remark we obtain:
\begin{co}
Let $f\in III$ with real coefficients $\delta^r$, $\delta$. The automorphism group $\Aut(H_f,s_f)$ can be identified with the  semi-direct product 
$$
\left[\qmod{\R^*\times\R^*}{\langle (\delta^r,\delta)\rangle}\right]\ltimes_{\hat \rho_r} \R.
$$	
\end{co} 

Similar descriptions are obtained for $f\in II_a$ and  $f\in II_b$. For $f\in II_c$ with real coefficients we have obviously  $\Aut(H_f,s_f)=(\R^*\times\R^*)/\langle(\alpha,\delta)\rangle$, and for $f\in II'_c$ we obtain  $\Aut(H_f,s_f)=\C^*/\langle  \alpha\rangle$, which is a 1-dimensional complex torus. 

\vspace{2mm}

The automorphism group of the odd Real  Hopf surfaces is given by the following
\begin{thry}\label{AutOdd} 
Let $f\in IV\cup III\cup II_a\cup II_b\cup II_c$	.
\begin{enumerate}
\item Suppose $f$ has real coefficients and positive diagonal coefficients. Then
$$
\Aut(H_f,\sigma_f)=\Aut(H_f,s_f).
$$
\item  Suppose $f\in II'_c$. Then again
$$
\Aut(H_f,\sigma_f)=\Aut(H_f,s_f).
$$
\item \label{neg-alpha} Suppose that $f\in IV$ with negative diagonal coefficient $\alpha$. Then
$$
\Aut(H_f,\sg_f)=\qmod{\left\{\begin{pmatrix} a &-\bar b\\ b & \bar a \end{pmatrix}\vline\ (a, b)\in \C^2\setminus\{0\} \right\} }{\langle \alpha I_2\rangle}.       
$$
\end{enumerate}
\end{thry}
 \begin{proof}
 Use  similar arguments, based on elementary computations, as in the proof of Theorem \ref{AutEven}.	
 \end{proof}

 The group $\left\{\begin{pmatrix} a &-\bar b\\ b & \bar a \end{pmatrix}\vline\ (a, b)\in \C^2\setminus\{0\} \right\}$ can be identified with the subgroup $\R^*_+\SU(2)$ of $\GL(2,\C)$. This subgroup is isomorphic with $\H^*$ via the map:

$$
 z+jw\mapsto  \begin{pmatrix} z &-\bar w\\ w & \bar z \end{pmatrix}.
 $$
  Therefore in case (\ref{neg-alpha}) we have
\begin{equation}\label{NewQuot} 
 \Aut(H_f,\sg_f)\simeq \qmod{\R^*_+\SU(2)}{\langle \alpha I_2\rangle}.
\end{equation}
 The right hand quotient in (\ref{NewQuot}) can be written as the quotient of ${\R^*_+  \SU(2)}/{\langle \alpha^2 I_2\rangle}$ by the order 2 group $\langle \alpha I_2\rangle/\langle \alpha^2 I_2\rangle$.   Via the Lie group isomorphism   
 $$\Phi: \R^*_+\SU(2)\to \R^*_+\times\SU(2), \  \Phi(A)=( \det(A)^{\frac{1}{2}},\det(A)^{-\frac{1}{2}}A)$$
 the matrices $\alpha I_2$, $\alpha^2I_2$ are mapped to $(|\alpha|,-I_2)$, $(|\alpha|^2,I_2)$ respectively. Therefore $\Phi$ induces an isomorphism 
$$
\phi:\qmod{\R^*_+  \SU(2)}{\langle \alpha^2 I_2\rangle}\textmap{\simeq} \qmod{\R^*_+ \times \SU(2)}{\langle (\alpha^2, I_2)\rangle}=\qmod{\R^*_+}{\langle \alpha^2\rangle}\times \SU(2),
$$
and the image of $\alpha I_2$ in the right hand group is $([|\alpha|],-I_2)$. Identifying ${\R^*_+}/{\langle \alpha^2\rangle}$ with $S^1$ via the isomorphisms $[\rho]\mapsto e^{\pi i\frac{\ln(\rho)}{\ln|\alpha|}}$,  and noting that the image of $[|\alpha|]$ in $S^1$  via this identification is $-1$, we   obtain an isomorphism 
$$\qmod{\R^*_+\SU(2)}{\langle \alpha I_2\rangle}\textmap{\simeq} \qmod{S^1\times \SU(2)}{\langle (-1,-I_2)\rangle}.$$
 Therefore:
  \begin{co}\label{Spinc(3)}
Let $\alpha\in (-1,0)$ and $f\edf \alpha I_2$. The automorphism group of the odd Real 	Hopf surface $(H_f,\sg_f)$ is naturally isomorphic to the group 
$$\Spin^c(3)=S^1\times_{\Z_2} \Spin(3)=S^1\times_{\Z_2} \SU(2).$$
\end{co}

\subsection{The Real Picard group of a Real primary  Hopf surface}

For a class VII surface $X$ the canonical Lie group morphism
\begin{equation}\label{isoPic}
\Hom (\pi_1(X,x_0),\C^*)=\Hom(H_1(X,\Z),\C^*)\to \Pic(X)	
\end{equation}
is injective and its image is the subgroup $\Pic^T(X)$ of isomorphism classes of holomorphic line bundles with torsion Chern class (see \cite[Remark 3.2.3]{Te}). For a primary Hopf surface $X$ we have $H^2(X,\Z)=\{0\}$, so  $\Pic^0(X)=\Pic^T(X)=\Pic(X)$, so (\ref{isoPic}) is an isomorphism. Since $\pi_1(X,x_0)\simeq \Z$, we obtain an isomorphism 
$$\lambda:\C^*=\Hom(\pi_1(X,x_0),\C^*)\to\Pic(X)$$
which can be obtained explicitly as follows. Suppose $X=H_f$ for a holomorphic contraction $f\in\Aut_h(W)$ (see section \ref{WehlClass}). For $\zeta\in\C^*$ put 
$$
L_\zeta\coloneq \qmod{W\times\C}{\langle f_\zeta\rangle},
$$
where $f_\zeta: W\times\C\to W\times\C$ is the fiberwise linear automorphism 
$$
f_\zeta (x,z)\coloneq (f(x),\zeta z).
$$
Endowed  with the obvious surjective submersion  $L_\zeta\to H_f$ given by 
$$[(x,z)]_{\langle f_\zeta\rangle}\mapsto [x]_{\langle f\rangle}\,,$$
$L_\zeta$ is naturally a holomorphic line bundle on $H_f$.  Recall (see for instance \cite[Section 2.2]{Te} that
\begin{re}\label{lambdaLzeta}
The map $ \lambda:\C^*\to \Pic (H_f)$  defined by $\lambda(\zeta)=[L_\zeta]$ is a Lie group isomorphism. 	
\end{re}

The following proposition shows that, for  any $f\in IV\cup III\cup II_a\cup II_b\cup II_c$ and any Real structure $s$ on $H_f$,  the anti-holomorphic involutive isomorphism 
$$\bar s^*: \Pic(H_f)\to \Pic(H_f)$$
 induced by  $s$ (see section \ref{intro}) is given by the same formula.  

\begin{pr}\label{sbar*}
Let $H$ be a primary Hopf surface, and let $s\in \Ah(H)$ (not necessarily involutive).   Then for any $\zeta\in \C^*$ we have
$$\bar s^*([L_\zeta])\simeq [L_{\bar \zeta}].
$$ 
\end{pr}
\begin{proof} 
By Proposition \ref{PrHolAnti}	 we know that $s$ is induced by an anti-holomorphic automorphism $\hat s\in \Ah(W)^f$. The commutative diagram
$$
\begin{tikzcd}[row sep=large, column sep=large ]
W\times\C \ar[r, " (\hat s{,}\bar{\ })"] \ar[d, "f_{\bar\zeta}=(f{,}\bar\zeta\cdot)"']&  W\times \C \ar[d, "f_{\zeta}=(f{,}  \zeta\cdot)"] \\
W\times\C \ar[r, " (\hat s{,}\bar{\ })"]&  W\times \C
\end{tikzcd}
$$
shows that the map $(\hat s{,}\bar{\ })$ descends to a well defined, fiberwise anti-linear, anti-holomorphic, $s$-lifting isomorphism $L_{\bar \zeta}\to L_{\zeta}$. The same map can be regarded as a fiberwise linear, holomorphic,   $s$-lifting isomorphism $L_{\bar \zeta}\to \bar L_{\zeta}$, where this time $s$ has been regarded as a holomorphic map $H_f\to  \bar H_f$. Therefore $L_{\bar \zeta}\simeq s^*(\bar  L_\zeta)=\bar s^*(L_\zeta)$.
\end{proof}

Proposition \ref{sbar*} shows in particular that, for any Real primary Hopf surface $(H,s)$, the associated Real structure  $\bar s^*$ on $\Pic(H)$  is always given, via the isomorphism $\lambda$, by the standard conjugation $\C^*\to\C^*$. \\

If follows that if a line bundle $L_\zeta$ on a Real Hopf surface $(H,s)$ admits an anti-holomorphic Real structure $\phi$, then $\zeta\in\R^*$. Let $\zeta\in\R^*$.  The proof of Proposition \ref{sbar*} shows that the map $(\hat s, \bar{\ })$ descends to a fiberwise anti-linear, anti-holomorphic, $s$-lifting isomorphism $\phi_0:L_\zeta\to L_\zeta$. Any fiberwise anti-linear, anti-holomorphic, $s$-lifting isomorphism $\phi:L_\zeta\to L_\zeta$ has the form $\phi=\nu \phi_0$ for a constant $\nu \in\C^*$. Indeed, the composition $\phi\circ \phi_0^{-1}$ is a holomorphic $\id_H$-covering automorphism of $L$, so   $\phi\circ \phi_0^{-1}=\nu \id_{L_ \zeta}$ with $\nu \in\C^*$.

We have 
\begin{equation}\label{phi2onL}
\phi^2=(\nu \phi_0)\circ (\nu \phi_0)=|\nu |^2\phi_0^2.	
\end{equation}
On the other hand $\phi_0^2$ is induced by the map $W\times\C\to W\times\C$ given by 
\begin{equation}\label{phi02onL}
W\times\C\ni (x,z)\mapsto (\hat s^2(x),z). 	
\end{equation}

\begin{itemize}
\item If $s$ is even, we can choose $\hat s$ such that $\hat s^2=\id_W$; in this case $\phi_0$ is already an anti-holomorphic Real structure on $L_ \zeta$ and formula (\ref{phi2onL}) shows that the set of all  anti-holomorphic Real structures on $L_ \zeta$ is $S^1\phi_0$.
\item 	If $s$ is odd, we can choose $\hat s$ such that $\hat s^2=f$; in this case formula (\ref{phi02onL}) shows that 
$$\phi^2_0([x,z])=[f(x),z]=[x,\zeta^{-1}z],$$
so $\phi^2_0=\zeta^{-1}\id_{L_\zeta}$.  Taking into account (\ref{phi2onL}) it follows that $\phi=\nu \phi_0$ is involutive if and only if $\zeta=|\nu|^2$.  Therefore, in this case, $L_\zeta$ admits anti-holomorphic Real structures if and only if $\zeta>0$, and if this is the case, the set of  anti-holomorphic Real structures on $L_\zeta$  is  $S^1\sqrt{\zeta}\phi_0$.
\end{itemize}

We have proved:
\begin{pr}\label{SetAHRealStrL}
Let $(H_f,s)$ be a Real primary Hopf surface, let $\hat s\in \Ah(W)^f$ be a lift of $s$ with $\hat s^2\in \{\id_W,f\}$ and let $\zeta\in \R^*$. 
\begin{enumerate}
\item If $\hat s^2=\id_W$ (i.e. if $s$ is even) then the set of anti-holomorphic Real structures on $L_\zeta$ is $S^1\phi_0$.
\item If $\hat s^2=f$ (i.e. if $s$ is odd) then $L_ \zeta$ admits  anti-holomorphic Real structures if and only if $\zeta>0$, and, if this the case,  the set of anti-holomorphic Real structures on $L_\zeta$ is $S^1\sqrt{\zeta}\phi_0$.
\end{enumerate}
\end{pr}

In all cases the group $\C^*\id_{L_\zeta}$ of holomorphic automorphisms acts on the set of  anti-holomorphic Real structures on $L_\zeta$ by conjugation, and the explicit  formula for this action is
$$
(z\id) \circ \phi \circ (z\id)^{-1}=z\bar z^{-1}\phi=({z}{|z|^{-1}})^2\phi.
$$
Since the map $\C^*\to S^1$ given by $z\mapsto   ({z}{|z|^{-1}})^2$ is obviously surjective, it follows that any $\psi\in S^1\phi$ is isomorphic (equivalent) to $\phi$.  This shows that, under the assumptions of Proposition \ref{SetAHRealStrL}, all anti-holomorphic Real structures  on $L_\zeta$ are isomorphic to either $\phi_0$ (if $\hat s^2=\id_W$), or to $\sqrt{\zeta}\phi_0$ (if $\hat s^2=f$ and $ \zeta>0$). \vspace{1mm}

In conclusion, we obtain:
\begin{pr}\label{Pic(X)R}
Let $(H,s)$ be a Real primary Hopf surface.  
\begin{enumerate}
\item Via the  isomorphism $ \C^*\textmap{\lambda}\Pic(H)$, the Real structure $ \Pic(H)\textmap{\bar s^*}\Pic(H)$ induced by $s$ coincides with the standard conjugation.

\item The map $[L_\zeta,\phi]\mapsto \zeta$ defines 
\begin{enumerate}
\item An isomorphism $\Pic_\R(H)\to \R^*$ if $s$ is an even Real structure.
\item An isomorphism $\Pic_\R(H)\to \R_{>0}$ if $s$ is an odd Real structure.
\end{enumerate}
\end{enumerate}
\end{pr}

The second statement shows that, for odd Real primary Hopf surfaces, the obvious group morphism $\Pic_\R(H)\to \Pic(H)(\R)$ is not surjective; the classes $[L_\zeta]$  with $\zeta\in\R_{<0}$ do not correspond to Real holomorphic line bundles in the sense of Definition \ref{RealBdlsHol}, although they are fixed points of $\bar s^*$.

\section{The differential-topological classification}

\subsection{The differential-topological classification of even Real Hopf surfaces }

On the product $S^1\times S^3$ we consider the following involutions 

\begin{equation}\label{model-involutions}
\begin{split}
\tau(\zeta,(u,v))&\edf (\zeta,(\bar u,\bar v))\\
\tau'(\zeta,(u,v))&\edf(\zeta,(\bar u,\zeta\bar v)).
\end{split}
\end{equation}

Our  goal is the following classification result:
\begin{thry}\label{ClassDiffEven}
Any even Real primary Hopf surface  $(H,\sigma)$ is diffeomorphic, as a  $\Z_2$-manifold, to either 	$(S^1\times S^3,\tau)$ or $(S^1\times S^3,\tau')$	.
\end{thry}

We will also need the following involutions on $S^1\times S^3$:
\begin{equation}
j'(\zeta, (u,v))\edf(-\zeta,(u,-v)), \ 	j''(\zeta, (u,v))\edf(-\zeta,(-u,-v)).
\end{equation}

The order 2 groups $\langle j'\rangle$,  $\langle j''\rangle$ act freely and properly discontinuously on $S^1\times S^3$, so the we obtain double covers
$$
q':S^1\times S^3\to Q'\edf \qmod{S^1\times S^3}{\langle j'\rangle}, 	\ q'':S^1\times S^3\to Q''\edf \qmod{S^1\times S^3}{\langle j''\rangle}.
$$
Note that $\tau$ commutes with $j'$ and $j''$, so we obtain induced involutions
$$
\theta': Q'\to Q',\ \theta'': Q''\to Q''
$$
induced by $\tau$.
We will need the following notation:
\begin{dt}\label{Rzeta}
For $\zeta\in S^1$ we denote by 	$R_\zeta\in \SO(2)\subset\GL(2,\C)$   the $2\times 2$ matrix which corresponds to $\zeta$ via the standard isomorphism $S^1\to\SO(2)$.
\end{dt}

\begin{lm}\label{Q'Q''} Consider the maps $a': S^1\times  S^3\to S^1\times  S^3$, $a'': S^1\times  S^3\to S^1\times  S^3$ defined by
$$
a'(\zeta,(u,v))\edf (\zeta^2, (u,\zeta v),\   a''(\zeta,(u,v))\edf (\zeta^2,R_\zeta(u,v)),
$$
\begin{enumerate}
\item $a'$ 	 is $\langle j'\rangle$-invariant and induces a diffeomorphism of $\Z_2$-spaces 
$$\hat a': (Q',\theta')\textmap{\simeq} (S^1\times S^3,\tau').$$
\item  $a''$ 	 is $\langle j''\rangle$-invariant and induces a diffeomorphism of $\Z_2$-spaces 
$$\hat a'': (Q'',\theta'')\textmap{\simeq}  (S^1\times S^3,\tau).$$
 \end{enumerate}
\end{lm}
\begin{proof}
It's easy to see that 	 $a'$ ($ a''$)	 is $\langle j'\rangle$-invariant (respectively $\langle j''\rangle$-invariant) and that the induced maps $\hat a':Q'\to S^1\times S^3$, $\hat a'':Q''\to S^1\times S^3$ are bijective  and verify $\tau'\circ \hat a'=\hat a'\circ \theta'$, $\tau\circ \hat a''=\hat a''\circ \theta''$.   On the other hand $\hat a'$, $\hat a''$ are local diffeomorphisms because $a'$, $a''$ have this property. Therefore $\hat a'$, $\hat a''$  are diffeomorphisms. 
\end{proof} 

The idea of the proof of Theorem \ref{ClassDiffEven} is the following: using our classification Theorem \ref{ClassEven} we may suppose that $(H,\sigma)=(H_f,s_f)$, where $f\in IV\cup III\cup  {II}_a\cup  {II}_b\cup II_c$ is either with real coefficients, or $f\in II'_c$ and $s_f$ is the canonical Real structure on $H_f$. We will show (see Propositions \ref{SignsCoeff}, \ref {ClassDiffII'c} proved below) that $(H_f,s_f)$ is equivariantly diffeomorphic  to $(S^1\times  S^3,\tau)$,  to $(Q',\theta')$ or to $(Q'',\theta'')$. Theorem \ref{ClassDiffEven} will then follow by Lemma \ref{Q'Q''}.

\begin{pr}\label{SignsCoeff}
Let $f\in IV\cup III\cup  {II}_a\cup  {II}_b\cup II_c$ be with real coefficients. 
\begin{enumerate}
\item If the diagonal coefficients of $f$ are positive, then $(H_f,s_f)$ is equivariantly diffeomorphic to $(S^1\times  S^3,\tau)$. 
\item If a single diagonal coefficients of $f$ is negative, then $(H_f,s_f)$ is equivariantly diffeomorphic to $(Q',\theta')$. 
\item If both diagonal coefficients of $f$ are negative, then $(H_f,s_f)$ is equivariantly diffeomorphic to $(Q'',\theta'')$. 	
\end{enumerate}	
\end{pr}

The proof of Proposition \ref{SignsCoeff} requires a preparation. 
Let $f\in II_a$, 
$$f\begin{pmatrix}z\\   w\end{pmatrix}=\begin{pmatrix}\delta^r z+w^r\\ \delta w\end{pmatrix}.$$
Suppose $\delta\in\R$. The second power $f^2$ of $f$ given by
$$f^2\begin{pmatrix}z\\   w\end{pmatrix}=\begin{pmatrix}
 \delta^{2r}z+  2\delta^{r}w^r\\ \delta^2 w\end{pmatrix}$$
 has always positive diagonal coefficients, but, unfortunately $f^2\not\in II_a$. Similar remark for $f\in II_b$. Therefore we will need the slightly larger classes $\widetilde{II}_a\supset II_a$, $\widetilde{II}_b\supset II_b$ defined by
$$
\widetilde{II}_a\edf \left\{ f:W\to W\vert \ f\begin{pmatrix}z\\   w\end{pmatrix}=\begin{pmatrix}\delta^r z+cw^r\\ \delta w\end{pmatrix}, \ 0<|\delta|<1, \ c\in \C \right\}\ \
$$
$$
\widetilde{II}_b\edf \left\{ f:W\to W\vert \ f\begin{pmatrix}z\\   w\end{pmatrix}=\begin{pmatrix}\alpha z+cw\\ \alpha w\end{pmatrix}, \ 0<|\delta|<1, \ c\in \C \right\}.
$$

We begin with the following remark which shows that, any contraction $f\in IV\cup III\cup II_a\cup II_b\cup II_c$ with real coefficients and positive diagonal coefficients  can be identified with the term $f^1$ of an explicit smooth 1-parameter group $(f^t)_{t\in\R}$ of holomorphic automorphisms of $W$. 
More precisely: 
\begin{re}
Let $f\in IV\cup III\cup \widetilde{II}_a\cup \widetilde{II}_b\cup II_c$ with real coefficients and positive diagonal coefficients. For $t\in\R$ we define $f^t\in\Aut_h(W)$ by the formula specified in the third column of the following table:
$$  \begin{array}  {|c|c|c|}
\hline &&\\ [-1em]
\hbox{The class of } f &  f\begin{pmatrix}
z\\w 	
\end{pmatrix}& f^{t}\begin{pmatrix}
z\\w 	
\end{pmatrix}   
 \\ && \\ [-1em]
 \hline &&\\ [-1em]
IV & \begin{pmatrix}\alpha z\\ \alpha w\end{pmatrix} & \begin{pmatrix}\alpha^{t} z\\ \alpha^{t} w\end{pmatrix}
\\ && \\ [-1em]
 \hline &&\\ [-1em]
III & \begin{pmatrix}\delta^r z\\ \delta w\end{pmatrix}  & \begin{pmatrix}
 \delta^{rt}z\\ \delta^{t}w 	
\end{pmatrix}   
\\ && \\ [-1em]
 \hline &&\\ [-1em]
\widetilde{II}_a  & \begin{pmatrix}\delta^r z+cw^r\\ \delta w\end{pmatrix} &  \begin{pmatrix}
 \delta^{rt}z+ct\delta^{r(t-1)}w^r
\\ 
\delta^{t}w	
\end{pmatrix}  
\\ && \\ [-1em]
 \hline &&\\ [-0.7em]
\widetilde{II}_b & \begin{pmatrix}\alpha z+cw\\ \alpha w\end{pmatrix} & \begin{pmatrix}
 \alpha^{t}z+ct\alpha^{t-1}w
\\
\alpha^{t}w	
\end{pmatrix}
\\ && \\ [-1em]
\hline &&\\ [-1em]
II_c &\begin{pmatrix}\alpha z\\ \delta w\end{pmatrix} & \begin{pmatrix}
 \alpha^{t}z
\\
\delta^{t}w	
\end{pmatrix}  
\\ [0.8em] 
\hline
\end{array}
$$

Then
\begin{enumerate}
 \item The family $(f^t)_{t\in\R}$ a 1-parameter  group of automorphisms of $W$, i.e. the map $\R\ni t\mapsto f^t\in\Aut_h(W)$ is a group morphism. 
 \item $f^1=f$.	
\end{enumerate}
 \end{re} 
\begin{proof}
 This follows by elementary computations. 	
\end{proof}

For a map $\eta:W\to \R$ and $Z=(z,w)\in W$ we define $\eta_Z:\R\to \R$ by
\begin{equation}\label{etaZ}
\eta_Z(t)\edf \eta (f^t(Z)).	
\end{equation}

Since $(f^t)_{t\in\R}$ is a 1-parameter group of diffeomorphisms we have the identity
\begin{equation} \label{eta'Z}
\eta'_Z(t)=	\eta'_{f^t(Z)}(0).
\end{equation}

 \begin{re}\label{eta-n}
Let   $B\in (0,\infty)$ and  $q\in \N^*$. Let $\eta_{B}^{q}:W\to (0,\infty)$ be the map defined by
$$
\eta_{B}^{q}(z,w)=|z|^{2}+B|w|^{2q}.
$$
\begin{enumerate}
\item 	$\eta_{B}^{q}$ is a submersion, in particular the fiber $\Sigma_{B}^{q}\edf (\eta_{B}^{q})^{-1}(1)$ is a smooth hypersurface of $W$.
\item The restriction  $n_B^q:\Sigma_B^q\to S^3$ of the normalization map $N:W\to S^3$, $N(Z)\edf \frac{1}{\|Z\|} Z$ to $\Sigma_{B}^{q}$ is a diffeomorphism which commmutes with the involutions $(z,w)\mapsto (z,-w)$, $(z,w)\mapsto (-z,-w)$, $(z,w)\mapsto (\bar z,\bar w)$.
\end{enumerate}	
\end{re}
\begin{proof}
The first claim follows by elementary computation. For the second, note first that for any $Z\in W$ the half-line $\R^*_+ Z$ intersects $\Sigma_{B}^{q}$ in a unique point, which will be denoted $\Ng_{B}^{q}(Z)$. Using the implicit function theorem it follows easily that the obtained map  	$\Ng_{B}^{q}:W\to \Sigma_{B}^{q}$ is smooth, so  the restriction $\ng^q_B\edf \Ng_{B}^{q}|_{S^3}:S^3\to \Sigma_{B}^{q}$ will also be smooth. It suffices to note that
$$
\ng_{B}^{q} \circ n_{B}^{q}=\id_{\Sigma_{B}^{q}},\  n_{B}^{q}\circ \ng_{B}^{q}=\id_{S^3}.
$$
\end{proof}

\begin{lm}\label{ft-F}
Let $f\in IV\cup III\cup \widetilde{II}_a\cup \widetilde{II}_b\cup II_c$ with real coefficients and positive diagonal coefficients.  Let $\eta:W\to (0,+\infty)$ and $C<0$ 
be given by the following table:
$$  \begin{array}  {|c|c|c|c|}
\hline &&\\ [-1em]
\hbox{The class of } f &  f\begin{pmatrix}
z\\w 	
\end{pmatrix}& \eta & C    
 \\ && \\ [-1em]
 \hline &&\\ [-1em]
IV & \begin{pmatrix}\alpha z\\ \alpha w\end{pmatrix} &  \eta^1_1 & 2\ln(\alpha)
\\ && \\ [-1em]
 \hline &&\\ [-1em]
III & \begin{pmatrix}\delta^r z\\ \delta w\end{pmatrix}  & \eta^1_1 & 2\ln(\delta)\\ && \\ [-1em]
 \hline &&\\ [-1em]
\widetilde{II}_a  & \begin{pmatrix}\delta^r z+cw^r\\ \delta w\end{pmatrix} &  \eta^{r}_{B} \hbox{ with} \ B\geq  c^2 \frac{1}{ r^2\delta^{2r}\ln(\delta)^2} &  r\ln(\delta) \\ && \\ [-1em]
 \hline &&\\ [-0.7em]
\widetilde{II}_b & \begin{pmatrix}\alpha z+cw\\ \alpha w\end{pmatrix} &  \eta^{1}_{B} \hbox{ with} \ B\geq c^2\frac{1}{\alpha^2\ln(\alpha)^2} &  \ln(\alpha)\\ && \\ [-1em]
\hline &&\\ [-1em]
II_c &\begin{pmatrix}\alpha z\\ \delta w\end{pmatrix} & \eta^1_1 & 2\max(\ln(\alpha),\ln(\delta))\\ [0.8em] 
\hline
\end{array}
$$
\begin{enumerate}
\item 	In each case and for any $Z\in W$, the map $\eta_Z$ satisfies the differential inequality
\begin{equation}\label{DiffIneq}
 \eta_Z'(t)\leq C \eta_Z(t).
\end{equation}
 In particular $\eta_Z$ is strictly decreasing and 
\begin{equation}\label{limits}
\lim_{t\to \infty} \eta_{Z}(t)=0,\ \lim_{t\to -\infty} \eta_{Z}(t)=+\infty.	
\end{equation}

\item Put $\Sigma\edf \eta^{-1}(1)$.  The map $F:\R\times \Sigma \to W$ defined by $F(t,Z)=f^t(Z)$ is a diffeomorphism.
\item  Endowing $W$ with  the conjugation $c$ and $\R\times \Sigma$ with the involution $\id\times c_\Sigma$ (where $c_\Sigma$ denotes the involution induced by  $c$ on $\Sigma$), $F$ is equivariant.     
\end{enumerate}

\end{lm}

\begin{proof}
(1) The proof of (\ref{DiffIneq}) is based on formula (\ref{eta'Z}). For $f\in IV\cup III\cup II_c$ we have $\eta(z,w)=|z|^2+|w|^2$ and the computation of $\eta_Z'(0)$ is very easy. For $f\in \widetilde{II}_a$ one obtains for any $\varepsilon>0$:

\begin{equation}
\begin{split}
\eta'_Z(0)&=2r\ln(\delta)\big(|z|^2+ \frac{\delta^{-r}}{r\ln(\delta)}c\Re(z{\bar w^r})+ B|w|^{2r}) \\ 	
&\leq 2r\ln(\delta)\left( \bigg( 1+ \frac{\delta^{-r}\varepsilon}{2r \ln(\delta)}\bigg)|z|^2+\bigg(B+ \frac{\delta^{-r}c^2}{2r \ln(\delta)\varepsilon}\bigg)|w^r|^2 \right).
\end{split}	
\end{equation}
We first choose $\varepsilon=\delta^r {r|\ln(\delta)|}$ and we note that, for $B\geq  \frac{c^2}{ r^2\delta^{2r}\ln(\delta)^2}$ we have
$$B+ \frac{\delta^{-r}c^2}{2r \ln(\delta)\varepsilon}\geq \frac{1}{2}B.
$$
The case $f\in II_b$ is similar.
The formulae (\ref{limits}) follow from (\ref{DiffIneq}) by integrating the inequality 
$\frac{\eta'}{\eta}\leq C$.
\vspace{2mm}\\
(2) For the injectivity of $F$ let $(t,Z)$, $(t',Z')\in \R\times\Sigma$ such that $F(t,Z)=F(t',Z')$. This implies  $f^{t-t'}(Z')=Z$. Applying $\eta$ on both sides we obtain $\eta_{Z'}(t-t')=1=\eta_{Z'}(0)$.  Since $\eta_{Z'}$ is strictly decreasing, this implies $t=t'$, so we also have $Z'=Z$.

For the surjectivity, its suffices to note that for any $Z\in W$ there exists $t\in\R$ such that $f^{-t}(Z)\in\Sigma$, which is equivalent to $\eta(f^{-t}(Z))=1$, i.e. $\eta_Z(-t)=1$. But (\ref{limits}) shows that $\eta_Z(\R)=(0,+\infty)$.

$F$ is obviously differentiable. To see that it is a diffeomorphism it suffices to prove that $F$ is a local diffeomorphism, i.e. that the differential $d_{(t,Z)} F$  is invertible for any $(t,Z)\in \R\times\Sigma$. 

For $t\in \R$ denote by $\tau_t:\R\to \R$ the translation by $t$, i.e. $\tau_t(s)=s+t$.  Note that   
$$
F\circ (\tau_t,\id_\Sigma)=f^t\circ F.
$$
This implies
$$
F_{* (s+t,Z)}=F_{*(\tau_t,\id_\Sigma)(s,Z)}\circ (\tau_t,\id_\Sigma)_{*(s,Z)}=f^t_{*F(s,Z)}\circ F_{*(s,Z)}.
$$
For $s=0$ we get
$$
F_{*(t,Z)} =f^t_{*Z}\circ F_{*(0,Z)}.
$$

Since $f^t$ is a diffeomorphism, $f^t_{*Z}$ is a linear isomorphism, so it suffices to prove  that $F_{*(0,Z)}$ is a linear isomorphism. Taking into account the dimensions, it suffices to prove that $\ker(F_{*(0,Z)})=0$.  But
\begin{equation}\label{F*(0,Z)}
F_{*(0,Z)}(h,v)=\frac{\partial F}{\partial t}(0,Z) h+ v, \ \forall (h,v)\in\R \times T_Z(\Sigma).
\end{equation} 

Let  $(h,v)\in\R \times T_Z(\Sigma)$ such that $F_{*(0,Z)}(h,v)=0$. It follows 
$$0=d\eta(F_{*(0,Z)}(h,v))=h\, d\eta(\frac{\partial F}{\partial t}(0,Z))+ d\eta(v).$$
Since $\eta|_{\Sigma}\equiv 1$ and $v\in T_Z(\Sigma)$ we have $d\eta(v)=0$. On the other hand
$$d\eta(\frac{\partial F}{\partial t}(0,Z))=\frac{d}{dt}|_{t=0}\,  \eta (F(t,Z))=\frac{d}{dt}|_{t=0}\,  \eta_Z (t),
$$
which is negative by (\ref{DiffIneq}).  Therefore $F_{*(0,Z)}(h,v)=0$ implies $h=0$. Coming back to (\ref{F*(0,Z)}) we obtain $v=0$.
\vspace{2mm}\\
(3)
Since $f^t$ has real coefficients, for any $t$, we have
$$F(t,\bar Z)=f^t(\bar Z)=\overline{f^t(Z)}=\overline{F(t, Z)}.
$$
\end{proof}

\begin{proof} (of Proposition \ref{SignsCoeff}) 

(1) By Lemma \ref{ft-F} the map $F:\R\times \Sigma \to W$ defined by $F(t,Z)=f^t(Z)$
is a diffeomorphism.   $F$ induces a diffeomorphism  
$$
\tilde F:(\R/\Z)\times \Sigma\to \qmod{W}{\langle f\rangle} = H_f
$$
 given by $\tilde F([t]_{\Z},Z)=[f^t(Z)]_{\langle f \rangle}$.

 Let $e:\R/\Z\to S^1$ be the standard diffeomorphism and $n:\Sigma\to S^3$ be the given by Remark \ref{eta-n}. We obtain the diffeomorphisms
 $$
 S^1\times S^3\textmap{ e^{-1}\times n^{-1}} (\R/\Z)\times \Sigma\textmap{\tilde F} H_f
 $$ 
 which (by Remark \ref{eta-n} (2) and Lemma \ref{ft-F} (3)) are equivariant with respect to the involutions $\tau$ on  $S^1\times S^3$, $\id\times c_\Sigma$   on $(\R/\Z)\times\Sigma$ and $s_f$ on $H_f$.
\vspace{2mm}

(2) Note first that, under  our assumption, we have either $f\in III\cup II_a$   with  $\delta<0$ and $r$ even (in which cases the second diagonal coefficient of $f$ is negative), or $f\in II_c$ with $\alpha\delta<0$. In the latter case we may suppose $\delta<0$ (see Theorem \ref {Wehl}(2)). Therefore we may always assume that the  second diagonal coefficient of $f$ is negative.

We apply Lemma \ref{ft-F} to $g=f^2$ which has real  coefficients and positive diagonal coefficients. Note that for $f\in II_a$ with diagonal coefficients $\delta^r$, $\delta$ we have $f^2\in  \widetilde{II}_a$ with diagonal coefficients $(\delta^2)^r$, $\delta^2$ and non-diagonal coefficient  $c=2\delta^r$.

The diffeomorphism $G:\R\times \Sigma \to W$,  $G(t,Z)=g^t(Z)$ given by Lemma \ref{ft-F} applied to $g$ induces a diffeomorphism 
$
\tilde  G:(\R/\Z)\times \Sigma\to H_g=H_{f^2}
$
(as above) which is equivariant with respect to the involutions $\id\times c_\Sigma$ and $s_g$.

Our primary Hopf surface  $H_f=W/\langle f\rangle $ is identified with the quotient   of $H_g$ by the involution $\hat f$ induced by $f$ on $H_g$,  which  is given explicitly by  
$$\hat f([Z]_{\langle g\rangle})=[f(Z)]_{\langle g\rangle}$$
and whose  fixed point locus is empty.  Let $\Jg':\R\times\Sigma\to \R\times\Sigma$ be the diffeomorphism  
$$
\Jg'(t,(u,v))\edf \big(t+\frac{1}{2},(u,-v)\big),
$$ 
and let $\jg':(\R/\Z)\times\Sigma\to  (\R/\Z)\times\Sigma$ be the {\it involution}  induced by $\Jg'$. Direct computations give
\begin{equation}
G\circ\Jg'=f\circ G,	 
\end{equation}
which obviously implies 
\begin{equation}
\tilde  G\circ\jg'=\hat f\circ \tilde G.	 
\end{equation}
Therefore $\tilde  G$ induces a diffeomorphism
$$
\hat G: \Qg'\edf \qmod{\R/\Z \times\Sigma}{\langle \jg'\rangle }\textmap{\simeq} \qmod{H_g}{\langle \hat f\rangle}=H_f
$$
(between the indicated free quotients) which is equivariant with respect to the following involutions:  $s_f$ on $H_f$ and  the involution  $\tg'$  induced by  $\id\times c_\Sigma$ on $\Qg'$.  
 
 It suffices to note that the diffeomorphism $e\times n:\R/\Z\times\Sigma\to S^1\times S^3$ induces  a  diffeomorphism $\Qg'\to  Q'$, which, by Remark \ref{eta-n} (2) is equivariant with respect to the involutions $\tg'$, $
 \theta'$.
 \vspace{2mm}\\
 (3)   We use similar arguments noting that  in this case 
 \begin{equation}
G\circ\Jg''=f\circ G,	 
\end{equation}
where $\Jg'':\R\times\Sigma\to \R\times\Sigma$  is   given by $\Jg''(t,(u,v))\edf \big(t+\frac{1}{2},(-u,-v)\big)$. Denoting  
by  $\Qg''$ the quotient of $\R/\Z \times\Sigma$ by the involution $\jg''$  induced by $\Jg''$, we  
obtain a  diffeomorphism $\Qg''\to Q''$ induced again by $e\times n$ which is equivariant with respect to the involutions $\tg''$ (defined similarly) and $\theta''$.

\end{proof}

Let $f\in II'_c$  be of the form
$$
f\begin{pmatrix}z\\w\end{pmatrix}=\begin{pmatrix}\alpha z\\\bar \alpha w\end{pmatrix}
$$
with $|\alpha|<1$, $\alpha\in\C\setminus\R$.
Consider the Real structure $c'$ on $W$ given in Example \ref{ExStandard}:
$$
c'\begin{pmatrix}z\\w\end{pmatrix}=\begin{pmatrix}\bar w\\ \bar z\end{pmatrix}.
$$

\begin{lm}
Let $\tau:(0,+\infty)\to \R$ be a ${\cal C}^\infty$ map. The map  
$$
\Psi_\tau:W\to W,\ \Psi_\tau(Z)\edf R_{e^{i\tau(\|Z\|)}} Z
$$
is a diffeomorphism.
\end{lm}
\begin{proof}
It suffices to note that $\Psi_\tau$ is obviously  differentiable  and that, for $\tau$, $\theta\in   {\cal C}^\infty((0,+\infty), \R)$, we have $\Psi_\tau\circ \Psi_\theta=\Psi_{\tau+\theta}$. It follows that $\Psi_\tau\circ \Psi_{-\tau}=\Psi_{-\tau}\circ \Psi_\tau=\id_W$, in particular $\Psi_\tau$ is bijective and its inverse is $\Psi_{-\tau}$, which is also differentiable.
\end{proof}

\begin{lm}\label{II'c-lg}
Let $\theta\in\R$ be such that by 	$
\frac{1}{|\alpha|}\alpha=e^{i\theta}$ and $\tau\in {\cal C}^\infty((0,+\infty), \R)$ given by $\tau(t)=\frac{\theta}{\ln|\alpha|}\ln(t)$. 
Let $L\edf \begin{pmatrix}1&i\\ 1 & - i\end{pmatrix}\in\GL(2,\C)$, $l:W\to W$ the associated  diffeomorphism, and $\lg\edf l\circ\Psi_\tau$.  Then 
\begin{enumerate}
\item We have  
$$l^{-1}\circ c'\circ l=c,\ l^{-1}\circ f\circ l=|\alpha| R_{\frac{1}{|\alpha|}\alpha}.$$
\item We have
$$\lg^{-1}\circ f\circ \lg=f_{|\alpha|},\ \lg^{-1}\circ c'\circ \lg=c.
$$
Therefore $\lg$ induces an equivariant diffeomorphism $ (H_{f_{|\alpha|}},s_{f_{|\alpha|}})\stackrel{\hat\lg}{\to}(H_f,s_f)$.
\end{enumerate}

\end{lm}
\begin{proof}
Direct computations
\end{proof}

Taking into account that $f_{|\alpha|}$ belongs to the class $IV$ and has positive diagonal coefficients, we obtain by Proposition \ref{SignsCoeff}(1):
\begin{pr}\label{ClassDiffII'c}
Any even  Real Hopf surface 	$(H_f,s_f)$ with $f\in II'_c$ is equivariantly diffeomorphic to $(S^1\times S^3,\tau)$.
\end{pr}

\subsection{The differential-topological classification of odd Real Hopf surfaces }

The goal of this section is the following classification theorem

\begin{thry}\label{ClassDiffOdd}
Every odd Real primary Hopf surface is equivariantly diffeomorphic to $(S^1\times S^3,\mu)$, where $\mu$ is the involution 
$$
\mu(\zeta, Z)\edf(-\zeta, \bar Z).
$$
\end{thry}

\begin{proof}
By the classification Theorem \ref{ClassOdd} we know that any odd Real primary Hopf surface is (equivariantly  biholomorphically) isomorphic to one of the following:
\begin{enumerate}
	\item $(H_f,\sigma_f)$, where $f\in IV\cup III\cup II_a\cup II_b\cup II_c$ has real coefficients and postive diagonal coefficients and $\sigma_f$ is induced by $c\circ f^{\frac{1}{2}}$,  where $f^{\frac{1}{2}}$ is defined in  Remark \ref{kroot}.
	\item $(H_f,\sigma_f)$, where $f\in  II'_c$ and $\sigma_f$ is induced by $c'\circ f^{\frac{1}{2}}$, where $f^{\frac{1}{2}}$ is defined in Proposition \ref{class-odd-II'c}.
	\item $(H_f,\sg_f)$, where $f\in IV$ has negative diagonal coefficient and $\sg$ is induced by $J\circ f^{\frac{1}{2}}$, where $f^{\frac{1}{2}}$ is defined in Proposition \ref{oddIV-neg-coeff}.
\end{enumerate}	
(1) In the first case note that the square root  we use the diffeomorphism 
$$
\tilde F:(\R/\Z)\times \Sigma\to \qmod{W}{\langle f\rangle} = H_f
$$
as in the proof of Proposition \ref{SignsCoeff}, and we note that the involution $\sigma_f$ on $H_f$ corresponds via $\tilde F$ to the involution 
$$
([t],Z)\mapsto ([t+\frac{1}{2}],\bar Z)
$$
on $(\R/\Z)\times \Sigma$, which corresponds to the  involution $\mu$ on $S^1\times S^3$ via $e\times n$.
\vspace{2mm}\\
(2)  Lemma \ref{II'c-lg} gives an equivariant  diffeomorphism $\hat\lg:  H_{f_{|\alpha|}} \to  H_f$ induced by $\lg:W\to W$. A direct computation gives  
$$
\lg\circ (c\circ f_{|\alpha|}^{\frac{1}{2}})\circ\lg^{-1}=c'\circ f^{\frac{1}{2}},
$$
which implies $\hat\lg\circ \sigma_{f_{|\alpha|}}\circ\hat\lg^{-1}=\sigma_f$. Therefore $(H_f,\sigma_f)$ is equivariantly diffeomorphic to $(H_{f_{|\alpha|}}, \sigma_{f_{|\alpha|}})$, which belongs to the class considered above.
\vspace{2mm}\\
(3) Let $f=f_\alpha$ with $\alpha\in (-1,0)$. As in the even case,  put $g\edf f^2=f_{\alpha^2}$, note that by Lemma \ref{ft-F}, the hypersurface $\Sigma$ associated with  $g$  coincides with $S^3$,  and consider the diffeomorphism $\tilde G:(\R/\Z)\times S^3\to H_g$  induced by $G:\R\times S^3\to W$.

 Let $\mg:\R\times S^3\to \R\times S^3$ be the diffeomorphism defined by 
$$
\mg (t,Z)\edf \big(t+\frac{1}{4},JiZ\big).$$
The induced map $\tilde \mg: (\R/\Z)\times S^3\to (\R/\Z)\times  S^3$ is a diffeomorphism of order 4 of $(\R/\Z)\times  S^3$  whose second power $\tilde\mg^2$ is given by the formula:
$$
\tilde \mg ^2([t],Z)=([t+\frac{1}{2}],-Z).
$$
Direct computations give
\begin{equation}\label{mJf1/2}
G\circ \mg =(J\circ f^{\frac{1}{2}})\circ G. 
\end{equation}

Recall that our primary Hopf surface  $H_f=W/\langle f\rangle $ is identified with the (free) quotient   of $H_g$ by the involution $\hat f$ induced by $f$ on $H_g$, and, via $\tilde G$, $\hat f$ corresponds to the involution 
$$
\jg'': (\R/\Z)\times  S^3\to (\R/\Z)\times  S^3,\ \jg''([t],Z)=([t+\frac{1}{2}],-Z).
$$
Therefore $\tilde G$ induces a diffeomorphism
$$
\hat G: \Qg''\edf \qmod{(\R/\Z)\times  S^3}{\langle\jg''\rangle}\to H_f.
$$
Formula (\ref{mJf1/2}) shows that, via $\hat G$, the involution $\sg_f$ on $H_f$ corresponds to the involution $\hat \mg $ induced by $\tilde \mg$ on $\Qg''$. Identifying $\R/\Z$ with $S^1$ and $\Qg''$ with $Q''$ in the usual way, we see that the involution $\hat m$ on $Q''$ which corresponds to $\hat \mg $ is given by
$$
\hat m([\zeta,Z])=[e^{i\frac{\pi}{2}}\zeta, JiZ].
$$

Via the diffeomorphism $\hat a''$ given by Lemma \ref{Q'Q''} the involution on $S^1\times S^3$ which corresponds to $\hat m$ is
$$
\mu':S^1\times S^3 \to S^1\times S^3,\ \mu'(\zeta, Z)=(-\zeta, i\bar Z).
$$
It suffices to note that $(\zeta,Z)\mapsto (\zeta,e^{-i\frac{\pi}{4}}Z)$ defines an equivariant diffeomorphism $(S^1\times S^3,\mu')\to (S^1\times S^3,\mu)$.

\end{proof}

\begin{re}\label{mu-mu0} Let $\mu_0:S^1\times S^3\to S^1\times S^3$ be the involution 
$$
\mu_0(\zeta,Z)=(-\zeta,Z).
$$
Let $\Phi$ be the $\R$-linear orientation preserving  isometry
$$
(z,w)\mapsto (\Re(z)+i\Re(w),\Im(z)+i\Im(w))
$$
and $\phi:S^3\to S^3$ the induced diffeomorphism of $S^3$. Let $\psi:S^1\times S^3\to S^1\times S^3$ be the diffeomorphism $(\zeta,(m,n))\mapsto (\zeta,(m,\zeta n))$. The composition $\psi\circ(\id\times\phi)$ is an equivariant diffeomorphism 
$$(S^1\times S^3,\mu)\to (S^1\times S^3,\mu_0).$$
  Therefore, in the classification Theorem \ref{ClassDiffOdd}, we may replace $\mu$ by $\mu_0$.
	
\end{re}
\begin{proof}
It suffices to note that, putting $\mu'(\zeta, (z,w) )\edf(-\zeta,(z,-w))$, we have
$$
(\id\times \phi)\circ \mu=\mu'\circ (\id\times \phi)
$$
and that $\mu_0\circ \psi=\psi\circ\mu'$. 	
\end{proof}

\subsection{The Real locus \texorpdfstring{$H^s$}{Hs} and the quotient \texorpdfstring{$H/\langle s\rangle$}{HHS}}

\subsubsection{The Real locus}\label{RealLocusSection}

Note first that the fixed point locus $M^\sigma$ of any involution $\sigma$ of a differentiable manifold $M$ is a submanifold of $M$. This follows by the equivariant  slice theorem (\cite[Theorem 5.6]{TTD}),	  which shows that an point $x\in X^\sigma$ has a $\sigma$-invariant open neighborhood which is equivariantly diffeomorphic to $(T_xX,\sigma_{*,x})$.

We have proved that any even Real structure on a primary Hopf surface $H_f$  with $f\in IV\cup III\cup II_a\cup II_b\cup II_c$ is equivalent to the standard Real structure $s_f$ (which is induced by the standard conjugation $c$ when $f$ is with real coefficients, and by $c'$ when $f\in II'_c$). Therefore, for describing the real locus of an arbitrary   even  Real primary Hopf surface, it suffices to consider only this standard Real structure.

 The fixed point locus $W^c$ (respectively $W^{c'}$) is 
$$
W^c=\R^2\setminus\{0\},\ W^{c'}=\{(z,\bar z)|\ z\in \C^*\}. 
$$
Let $f^{c}$, respectively $f^{c'}$ be the contraction induced by $f$ on $W^c$ (respectively $W^{c'}$).  The canonical maps
$$
\qmod{W^c}{\langle f^{c}\rangle }\to \left[\qmod{W}{\langle f\rangle}\right]^{s_f}=H_f^{s_f},\  \qmod{W^{c'}}{\langle f^{c'}\rangle }\to \left[\qmod{W}{\langle f\rangle}\right]^{s_f}=H_f^{s_f}
$$
are obviously diffeomorphisms. 

\begin{pr}
The fixed point locus $H_f^{s_f}$ of a standard even Real primary Hopf surface is 	diffeomorphic either to the torus $T^2$ or to a Klein bottle according to the sign of the determinant of the real part  of $f$. In particular, if $f\in II'_c$, then $H_f^{s_f}$ is a torus.
\end{pr}
\begin{proof}
By the classification Theorem 	\ref{ClassDiffEven} we have only two equivariant diffeomorphism classes of even Real primary Hopf surface. The proof of this theorem  shows  any  even Real primary Hopf surface  is equivariantly diffeomorphic   to either  $(H_g, s_g)$  or   $(H_h, s_h)$, where $g$, $h:W\to W$ are the contractions defined by 
$$
g(z,w)=\big(\frac{1}{2}z, \frac{1}{2}w),\  h(z,w)=\big(\frac{1}{2}z, -\frac{1}{2}w).
$$
Moreover, this proof also shows that if $f\in II'_c$, then $(H_f, s_f)$ is equivariantly diffeomorphic   to $(H_g, s_g)$.  It is easy to see that the quotient   $W^c/\langle g^c \rangle$ is a torus, whereas the quotient $W^c/\langle h^c\rangle$ is a Klein bottle.

\end{proof}

Note that, for $f\in II'_c$, the projections $(z,\bar z)\mapsto z$,  $(z,\bar z)\mapsto \bar z$ induce identifications 
$$\qmod{W^{c'}}{\langle f^{c'}\rangle}\textmap{\simeq} \qmod{\C^*}{\langle \alpha\rangle},\  \qmod{W^{c'}}{\langle f^{c'}\rangle}\textmap{\simeq} \qmod{\C^*}{\langle \bar \alpha\rangle},$$
where $\alpha$, $\bar\alpha$ are the coefficients of $f$. This shows that, in this case, the Real locus $H_f^{s_f}$ comes with a canonical (non-oriented) conformal structure, which is conformally isomorphic to the elliptic curve $E_\alpha\edf {\C^*}/{\langle \alpha\rangle}$. In other words:
\begin{re}
 When $f$ belongs to  the subclass $II'_c$, the real locus $H_f^{s_f}$ is a 2-torus which comes with a natural (non-oriented) conformal structure. 	
\end{re}

For the odd Real structures we have:
\begin{re}
The Real locus of any  odd Real primary Hopf surface	 is empty.
\end{re} 

\begin{proof}
Let $(H_f,\sigma)$ be an odd Real Hopf surface and $\hat\sigma\in Aut_h(W)$	be a lift of $\sigma$ to $W$ such that $\hat\sigma^2=f$. Suppose that $x=[(z,w)]\in  H_f$ is a fixed point of $\sigma$. Therefore there exists $k\in \Z$ such that $\hat\sigma(z,w)=f^k(z,w)$. Therefore $\hat\sigma^{2k-1}(z,w)=(z,w)$, which implies   $\hat\sigma^{2(2k-1)}(z,w)=\hat\sigma^{2k-1}(z,w)=(z,w)$. We obtain $f^{2k-1}(z,w)=(z,w)$. Since $\langle f\rangle$ acts freely on $W$, it follows $2k-1=0$ (contradiction). 
\end{proof}

\subsubsection{The quotient of a Real Hopf surface by its involution}

Note first that the quotient $X/\langle \sigma\rangle$ associated with {\it any} Real complex surface $(X,\sigma)$  is  a topological 4 manifold. This follows using
\begin{enumerate}
\item  The equivariant  slice theorem quoted above.
\item The classification of anti-linear involutions on a complex vector space (see for instance Remark \ref{linearRealRem} in this article).
\item The homeomorphism $\R^2/\langle -\id_{\R^2}\rangle\simeq\R^2$ induced (for instance) by the $-\id_{\R^2}$-invariant  map 
$$\beta:\R^2\to\R^2,\ \beta(\rho\cos(\theta),\rho\sin(\theta))=(\rho^2\cos(2\theta),\rho^2\sin(2\theta)).$$
Using a complex coordinate $\zeta$ on $\R^2$, this map is given by $\zeta\mapsto \zeta^2$.
\end{enumerate}

Taking into account this remark we will describe the quotient associated with a Real  primary Hopf surface as a topological 4-manifold.

We have seen that any even Real primary Hopf surface is isomorphic to $(H_f,s_f)$, where either $f\in IV\cup III\cup II_a\cup II_b\cup II_c$ with real coefficients, or $f\in II'_c$. The involution $s_f$ is induced by the {\it anti-linear} involution  $c$, respectively $c'$ on $W$.

The quotient
$$
Q_f\edf \qmod{H_f}{\langle s_f\rangle }
$$
can be identified with the quotient $\Wg\edf W/\langle c	\rangle$, respectively $\Wg'\edf W/\langle c'	\rangle$ by the contraction $\fg$ induced by $f$ on $\Wg$, respectively $\Wg'$. 

We obtain  a  decomposition  $\C^2=V_+\oplus V_-$ of $\C^2$ as direct sum of $c$-invariant (respectively  $c'$-invariant) 2-dimensional real linear subspaces $V_\pm$ such that $c|_{V_\pm }=\pm\id_{V_\pm}$ (respectively $c'|_{V_\pm }=\pm\id_{V_\pm}$).

Therefore the quotient $\C^2/\langle c\rangle$ ($\C^2/\langle c'\rangle$) can be identified with  the quotient
$$V_+\times \qmod{V_-}{\langle-\id_{V_-}\rangle}\simeq \R^4.$$

Here we have used   linear isomorphisms $V_\pm\simeq\R^2$ and the homeomorphism 
$$\R^2/\langle-\id_{\R^2}\rangle\to \R^2$$
 induced by $\beta$. The image of $W/\langle c\rangle$ ($W/\langle c'\rangle$) via this homeomorphism is $\R^4\setminus\{0\}$ and the image of the fixed point locus $W^c$ ($W^{c'}$) in $\Wg$ ($\Wg'$)  is $(\R^2\setminus\{0\})\times\{0\}$. Using the classification Theorem \ref{ClassDiffEven} we see that in all cases the contraction ${\cal F}$ induced by $f$ on $W/\langle c\rangle$ ($W/\langle c'\rangle$) is orientation preserving. 
 Therefore we obtain the following result, which describes the quotient $Q_f$ associated with an even Real primary Hopf surface, as well  as the image of the fixed point locus $H_f^{s_f}$ in this quotient:
\begin{pr} With the notations above we have
\begin{enumerate}
\item 	The quotient $Q_f\edf  {H_f}/{\langle s_f\rangle}$ can be  identified topologically with the quotient 
$$
{\cal Q}_f=\qmod{\R^4\setminus\{0\}}{\langle {\cal F}\rangle }
$$
of $\R^4\setminus\{0\}$ by the cyclic group generated by an  orientation preserving  contraction ${\cal F}$.
Therefore, in all cases   $Q_f$ is homeomorphic to $S^1\times S^3$.
\item ${\cal F}$ leaves invariant the punctured plane $(\R^2\setminus\{0\})\times\{0\}$ and,  via the above identification, the fixed  point locus $H_f^{s_f}$ corresponds to the quotient of $\R^2\setminus\{0\}$ by the contraction ${\cal F}_0$ induced by ${\cal F}$. 
\item ${\cal F}_0$ is orientation preserving if and only if the diagonal coefficients of $f$ have the same sign. If this is the case $\R^2\setminus\{0\}/\langle {\cal F}_0\rangle$ is a 2-torus. If the diagonal coefficients of $f$ have opposite signs, $\R^2\setminus\{0\}/\langle {\cal F}_0\rangle$ is a Klein bottle.
\end{enumerate}

\end{pr}

Taking into account the classification Theorem \ref{ClassDiffOdd} and Remark \ref{mu-mu0} we obtain the following simple description of the quotient $H/\langle\sigma\rangle$ of any odd Real primary Hopf surface $(H,\sigma)$:
\begin{pr}
Let $(H,\sigma)$ be an 	 odd Real primary Hopf surface. The quotient $H/\langle\sigma\rangle$ can be identified with $S^1\times S^3$ and the canonical projection
$$
H\to H/\langle\sigma\rangle
$$
is a double cover whose (non-trivial) deck transformation is an anti-holomorphic involution.
\end{pr}


\end{document}